\documentclass[12pt]{amsart}

\usepackage{amsaddr}
\usepackage{amssymb}
\usepackage{extarrows}
\usepackage{geometry}
\usepackage{graphicx}
\usepackage{subcaption}
\usepackage[table]{xcolor}

\graphicspath{{pictures/}}

\newcommand{\ju}[1]{[\![{#1}]\!]}
\newcommand{\jut}[1]{[\![{#1}]\!]_{t}}
\newcommand{\vct}[1]{\mathbf{#1}}

\newcommand{\cC}{\mathcal{C}}
\newcommand{\curl}{\mathrm{curl}}
\newcommand{\dvg}{\mathrm{div}}
\newcommand{\dx}{\mathrm{d}{\vx}}
\newcommand{\dy}{\mathrm{d}{\vy}}
\newcommand{\dz}{\mathrm{d}{\vz}}

\newcommand{\Ehin}{\mathcal{E}_{h}^{I}}

\newcommand{\Fhin}{\mathcal{F}_{h}^{I}}
\newcommand{\jac}{J}
\newcommand{\jhdel}{\hat{\j}^{\Delta}}
\newcommand{\Nd}{R}
\newcommand{\Qh}{\mathcal{Q}_h}
\newcommand{\RT}{D}
\newcommand{\Th}{\mathcal{T}_h}
\newcommand{\vG}{\vct{G}}
\newcommand{\vH}{\vct{H}}
\newcommand{\vhGdel}{\hat{\vG}^{\Delta}}
\newcommand{\vhHdel}{\hat{\vH}^{\Delta}}
\newcommand{\vhjdel}{\hat{\vjj}^{\Delta}}
\newcommand{\vht}{\hat{\vct{t}}}
\newcommand{\vj}{\vct{j}}
\newcommand{\vjj}{\boldsymbol{\j}}
\newcommand{\vjdel}{\vct{j}^{\Delta}}
\newcommand{\vn}{\hat{\vct{n}}}

\newcommand{\vR}{\vct{R}}
\newcommand{\vtH}{\tilde{\vH}}
\newcommand{\vtHdel}{\vtH^{\Delta}}
\newcommand{\vtht}{\boldsymbol{\theta}}
\newcommand{\vu}{\vct{u}}
\newcommand{\vv}{\vct{v}}
\newcommand{\vw}{\vct{w}}
\newcommand{\vx}{\vct{x}}
\newcommand{\vy}{\vct{y}}
\newcommand{\vz}{\vct{z}}
\newcommand{\vzero}{\vct{0}}

\newtheorem{thm}{Theorem}[section]
\newtheorem{cor}[thm]{Corollary}

\newtheorem{rem}[thm]{Remark}

\title[An a posteriori error estimator for N\'ed\'elec elements]{An equilibrated a posteriori error estimator for arbitrary-order N\'ed\'elec elements for magnetostatic problems}
\author{Joscha Gedicke$^{1*}$, Sjoerd Geevers$^{2*}$, Ilaria Perugia$^{2*}$}
\address{
$^1$ Institute for Numerical Simulation, University of Bonn \\
Endenicher Allee 19b, 53115 Bonn, Germany \\
$^2$ Faculty of Mathematics, University of Vienna\\
Oskar-Morgenstern-Platz 1, 1090 Vienna, Austria 
}
\thanks{*The first author has been funded by the Austrian Science Fund (FWF) through the project M~2640-N32. The second and third authors have been funded by FWF through the project F~65 ``Taming Complexity in Partial Differential Systems''. The first and third authors have also been funded by the FWF through the project P~29197-N32.}

\begin{document}

\maketitle

\begin{abstract}
We present a novel \textit{a posteriori} error estimator for N\'ed\'elec elements for magnetostatic problems that is constant-free, i.e. it provides an upper bound on the error that does not involve a generic constant.The estimator is based on equilibration of the magnetic field and only involves small local problems that can be solved in parallel. Such an error estimator is already available for the lowest-degree N\'ed\'elec element [D. Braess, J. Sch\"oberl, \textit{Equilibrated residual error estimator for edge elements}, Math. Comp. 77 (2008)] and requires solving local problems on vertex patches. The novelty of our estimator is that it can be applied to N\'ed\'elec elements of arbitrary degree. Furthermore, our estimator does not require solving problems on vertex patches, but instead requires solving problems on only single elements, single faces, and very small sets of nodes. We prove reliability and efficiency of the estimator and present several numerical examples that confirm this.
\end{abstract}
\medskip

{\footnotesize
}
\medskip

\section{Introduction}
We consider an \textit{a posteriori} error estimator for finite element methods for solving equations of the form $\nabla\times\mu^{-1}(\nabla\times\vu)=\vj$. These equations are related to magnetostatics but also appear in eddy current models for non-conductive media. 

The first \textit{a posteriori} error estimator in this context was introduced and analysed in \cite{beck00}. It is a residual-type estimator and provides bounds of the form
\begin{align*}
c_0\, \text{estimator} \leq \text{error} &\leq c_1\,\text{estimator},
\end{align*}
up to some higher-order data oscillation terms, where $c_0,c_1$ are positive constants that do not depend on the mesh resolution. Similar bounds can be obtained by hierarchical error estimators; see, e.g., \cite{beck99}, under the assumption of a saturation condition, and by Zienkiewicz--Zhu-type error estimators; see, e.g., \cite{nicaise05}. A drawback of these estimators is that the constants $c_0$ and $c_1$ are usually unknown, resulting in significant overestimation or underestimation of the real error.  

Equilibration-based error estimators can circumvent this problem. Often attributed to Prager and Synge \cite{prager47}, these estimators have become a major research topic; for a recent overview, see, for example, \cite{ern15} and the references therein. An equilibration-based error estimator was introduced for magnetostatics in \cite{braess08} and provides bounds of the form
\begin{align*}
c_0\, \text{estimator} \leq \text{error} &\leq \text{estimator}
\end{align*} 
up to some higher-order data oscillation terms. In other words, it provides a constant-free upper bound on the error. A different equilibration-based error estimator for magnetostatics was introduced in \cite{tang13} and, for an eddy current problem, in \cite{creuse17,creuse19b}. Constant-free upper bounds are also obtained by the functional estimate in \cite{neittaanmaki10}, when selecting a proper function $y$ in their estimator, and by the recovery-type error estimator in \cite{cai16}, in case the equations contain an additional term $\beta\vu$, with $\beta>0$. 

A drawback of the estimators in \cite{tang13,creuse17,creuse19b,neittaanmaki10} is that they require
solving a global problem. The estimator in \cite{braess08}, on the other hand, only involves solving local problems related to vertex patches. However, the latter estimator is defined for N\'ed\'elec elements of the lowest degree only.
In this paper, we present a new equilibration-based constant-free error estimator that can be applied to N\'ed\'elec elements of arbitrary degree. Furthermore, our estimator involves solving problems on only single elements, single faces, and very small sets of nodes.

The paper is constructed as follows: We firstly introduce a finite element method for solving magnetostatic problems in Section~\ref{sec:intro}. We then derive our error estimator step by step in Section~\ref{sec:errEst}, with a summary given in Section~\ref{sec:errEstOv}, and prove its reliability and efficiency in Section~\ref{sec:eff}. Numerical examples confirming the reliability and efficiency of our estimator are presented in Section~\ref{sec:numerics}, and an overall summary is given in Section~\ref{sec:conclusion}.

\section{A finite element method for magnetostatic problems}
\label{sec:intro}
Let $\Omega\subset\mathbb{R}^3$ be an open, bounded, simply connected, polyhedral domain with a connected Lipschitz boundary $\partial\Omega$. In case of a linear, isotropic medium and a perfectly conducting boundary, the static magnetic field $\vH:\Omega\rightarrow\mathbb{R}^3$ satisfies the equations
\begin{align*}
\nabla\times\vH &= \vj &&\text{in }\Omega, \\
\nabla\cdot\mu\vH &= 0 &&\text{in }\Omega, \\
\vn\cdot\mu\vH &= 0 &&\text{on }\partial\Omega,
\end{align*}
where $\nabla$ is the vector of differential operators $(\partial_1,\partial_2,\partial_3)$, $\times$ and $\cdot$ denote the outer- and inner product, respectively, (therefore, $\nabla\times$ and $\nabla\cdot$ are the curl- and divergence operator, respectively), $\vn$ denotes the outward pointing unit normal vector, $\mu:\Omega\rightarrow\mathbb{R}^+$, with $\mu_0\leq\mu\leq\mu_1$ for some positive constants $\mu_0$ and $\mu_1$, is a scalar magnetic permeability, and $\vj:\Omega\rightarrow\mathbb{R}^3$ is a given divergence-free current density. The first equality is known as Amp\`ere's law and the second as Gauss's law for magnetism. 

These equations can be solved by writing $\vH=\mu^{-1}\nabla\times\vu$, where $\vu:\Omega\rightarrow\mathbb{R}^3$ is a vector potential, and by solving the following problem for $\vu$:
\begin{subequations}
\label{eq:curlcurl}
\begin{align}
\nabla\times(\mu^{-1}\nabla\times \vu) &= \vj &&\text{in }\Omega, \\ 
\nabla\cdot\vu &= 0 &&\text{in }\Omega, \\
\vn\times\vu &= \vzero &&\text{on }\partial\Omega.
\end{align}
\end{subequations} 
The second condition is only added to ensure uniqueness of $\vu$ and is known as Coulomb's gauge. 

Now, for any domain $D\in\mathbb{R}^n$, let $L^2(D)^m$ denote the standard Lebesque space of square-integrable vector-valued functions $\vu:D\rightarrow\mathbb{R}^m$ equipped with norm $\|\vu\|_D^2:=\int_D \vu\cdot\vu \;\dx$ and inner product $(\vu,\vw)_D=\int_D \vu\cdot\vw\;\dx$, and define the following Sobolev spaces:
\begin{align*}
H^1(\Omega) &:= \{\phi\in L^2(\Omega) \;|\; \nabla \phi\in L^2(\Omega)^3 \}, \\
H_0^1(\Omega) &:= \{\phi\in H^1(\Omega) \;|\; \phi=0 \text{ on }\partial\Omega \}, \\
H(\curl; \Omega) &:= \{\vu\in L^2(\Omega)^3 \;|\; \nabla\times\vu\in L^2(\Omega)^3 \}, \\
H_0(\curl; \Omega) &:= \{\vu\in H(\curl;\Omega) \;|\; \vn\times\vu=\vzero \text{ on }\partial\Omega \}, \\
H(\dvg; \Omega) &:=  \{\vu\in L^2(\Omega)^3 \;|\; \nabla\cdot\vu\in L^2(\Omega) \}, \\
H(\dvg^0; \Omega) &:=  \{\vu\in H(\dvg;\Omega) \;|\; \nabla\cdot\vu=0 \}.
\end{align*}
The weak formulation of problem \eqref{eq:curlcurl} is finding $\vu\in H_0(\curl;\Omega)\cap H(\dvg^0;\Omega)$ such that
\begin{align}
\label{eq:WF}
(\mu^{-1}\nabla\times \vu,\nabla\times\vw)_{\Omega} &= (\vj,\vw)_{\Omega} &&\forall \vw\in H_0(\curl;\Omega),
\end{align}
which is a well-posed problem \cite[Theorem 5.9]{hiptmair02}.

The solution of the weak formulation can be approximated using a finite element method. Let $T$ be a tetrahedron and define $P_k(T)$ to be the space of polynomials on $T$ of degree $k$ or less. Also, define the N\'ed\'elec space of the first kind $\Nd_k(T)$ and the Raviart-Thomas space $\RT_k(T)$ by
\begin{align*}
\Nd_k(T) &:= \{\vu\in P_k(T)^3 \;|\; \vu(\vx) = \vv(\vx) + \vx\times\vw(\vx) \text{ for some }\vv,\vw\in P_{k-1}(T)^3\}, \\
\RT_k(T) &:= \{\vu\in P_k(T)^3 \;|\; \vu(\vx) = \vv(\vx) + \vx w(\vx) \text{ for some }\vv\in P_{k-1}(T)^3, w\in P_{k-1}(T)\}.
\end{align*}
Finally, let $\Th$ denote a tessellation of $\Omega$ into tetrahedra with a diameter smaller than or equal to $h$, let $P_k^{-1}(\Th)$, $\Nd_k^{-1}(\Th)$, and $\RT_k^{-1}(\Th)$ denote the discontinuous spaces given by
\begin{align*}
P_k^{-1}(\Th) &:= \{\phi\in L^2(\Omega) \;|\; \phi|_T \in P_k(T) \text{ for all }T\in\Th \}, \\
\Nd_k^{-1}(\Th) &:= \{\vu\in L^2(\Omega)^3 \;|\; \vu|_T \in \Nd_k(T) \text{ for all }T\in\Th \}, \\
\RT_k^{-1}(\Th) &:= \{\vu\in L^2(\Omega)^3 \;|\; \vu|_T \in \RT_k(T) \text{ for all }T\in\Th \}, 
\end{align*}
and define 
\begin{align*}
P_k(\Th)&:=P^{-1}_k(\Th)\cap H^1(\Omega), &  P_{k,0}(\Th)&:=P^{-1}_k(\Th)\cap H^1_0(\Omega),\\
\Nd_k(\Th)&:=\Nd^{-1}_k(\Th)\cap H(\curl;\Omega),  & \Nd_{k,0}(\Th)&:=\Nd^{-1}_k(\Th)\cap H_0(\curl;\Omega),\\
\RT_k(\Th)&:=\RT^{-1}_k(\Th)\cap H(\dvg;\Omega).  &
\end{align*}%
We define the finite element approximation for the magnetic vector potential as the vector field $\vu_h\in \Nd_{k,0}(\Th)$ that solves
\begin{subequations}
\label{eq:FEM}
\begin{align}
(\mu^{-1}\nabla\times \vu_h,\nabla\times\vw)_{\Omega} &= (\vj,\vw)_{\Omega} &&\forall \vw\in \Nd_{k,0}(\Th), \label{eq:FEMa}\\
(\vu_h,\nabla\psi)_{\Omega} &=0, &&\forall \psi\in P_{k,0}(\mathcal{T}_h). \label{eq:FEMb}
\end{align}
\end{subequations}
The approximation of the magnetic field is then given by
\begin{align*}
\vH_h:=\mu^{-1}\nabla\times\vu_h,
\end{align*}
which converges quasi-optimally as the mesh width $h$ tends to zero \cite[Theorem 5.10]{hiptmair02}. 

In the next section, we show how we can obtain a reliable and efficient estimator for $\|\vH-\vH_h\|_{\Omega}$.

\section{An equilibration-based a posteriori error estimator}
\label{sec:errEst}

We follow~\cite{braess08} and present an \textit{a posteriori} error estimator that is based on the following result.
\begin{thm}[{\cite[Thm.~10]{braess08}}]
\label{thm:estimator}
Let $\vu$ be the solution to (\ref{eq:WF}), let $\vu_h$ be the solution of (\ref{eq:FEM}), and set $\vH:=\mu^{-1}\nabla\times\vu$ and $\vH_h:=\mu^{-1}\nabla\times\vu_h$. If $\vtH\in H(\curl;\Omega)$ satisfies the equilibrium condition
\begin{align}
\label{eq:vtH}
\nabla\times\vtH &= \vj,
\end{align}
then
\begin{align}
\label{eq:errEst1}
\|\mu^{1/2}(\vH-\vH_h)\|_{\Omega} &\leq \|\mu^{1/2}(\vtH -\vH_h)\|_{\Omega}.
\end{align}
\end{thm}

\begin{proof}
The result follows from the orthogonality of $\mu^{1/2}(\vtH-\vH)$ and $\mu^{1/2}(\vH-\vH_h)$:
\begin{align*}
\big(\mu^{1/2}(\vtH-\vH),\mu^{1/2}(\vH-\vH_h)\big)_{\Omega} &=\big(\mu^{1/2}(\vtH-\vH), \mu^{-1/2}\nabla\times(\vu-\vu_h)\big)_{\Omega} \\ 
&= \big(\vtH-\vH,\nabla\times(\vu-\vu_h) \big)_{\Omega} \\ 
&= \big(\nabla\times(\vtH-\vH),\vu-\vu_h \big)_{\Omega} \\ 
&= (\vj-\vj, \vu-\vu_h)_{\Omega} \\
&= 0
\end{align*}
and Pythagoras's theorem
\begin{align}
\label{eq:pythagoras}
\|\mu^{1/2}(\vtH-\vH_h)\|^2_{\Omega} &= \|\mu^{1/2}(\vtH -\vH)\|^2_{\Omega} + \|\mu^{1/2}(\vH -\vH_h)\|^2_{\Omega}.
\end{align}
\end{proof}

\begin{rem}
Equation (\ref{eq:pythagoras}) is also known as a Prager--Synge type equation and obtaining an error estimator from such an equation is also known as the hypercircle method. Furthermore, equation (\ref{eq:vtH}) is known as the equilibrium condition and using the numerical approximation $\vH_h$ to obtain a solution to this equation is called equilibration of $\vH_h$.
\end{rem}

\begin{cor}
\label{cor:estimator}
Let $\vu$ be the solution to (\ref{eq:WF}), let $\vu_h$ be the solution of (\ref{eq:FEM}), set $\vH:=\mu^{-1}\nabla\times\vu$ and $\vH_h:=\mu^{-1}\nabla\times\vu_h$, and let $\vj_h:=\nabla\times\vH_h$ be the discrete current distribution. If $\vtHdel\in L^2(\Omega)^3$ satisfies the (residual) equilibrium condition
\begin{align}
\label{eq:vtHdel}
\nabla\times \vtHdel &= \vj-\vj_h,
\end{align}
which is an identity of distributions, then
\begin{align}
\label{eq:errEst2}
\|\mu^{1/2}(\vH-\vH_h)\|_{\Omega} &\leq \|\mu^{1/2}\vtHdel\|_{\Omega}.
\end{align}
\end{cor}

\begin{proof}
Since (\ref{eq:vtHdel}) is an identity of distributions, we can equivalently write
\begin{align*}
\langle \nabla\times\vtHdel,\vw \rangle &= \langle \vj-\vj_h, \vw \rangle &&\forall\vw\in\cC_0^{\infty}(\Omega)^3,
\end{align*}
where $\langle\cdot,\cdot\rangle$ denotes the application of a distribution to a function in $\cC_0^{\infty}(\Omega)^3$. Now, set $\vtH:=\vtHdel+\vH_h \in L^2(\Omega)$. Using the definition $\vj_h:=\nabla\times\vH_h$, we obtain
\begin{align*}
\langle \nabla\times\vtH,\vw \rangle &= \langle \nabla\times\vtHdel + \nabla\times\vH_h,\vw \rangle &&\\
& = \langle \vj-\vj_h + \nabla\times\vH_h, \vw \rangle && \\
&= \langle \vj, \vw \rangle &&\forall\vw\in\cC_0^{\infty}(\Omega)^3.
\end{align*}
From this, it follows that $\nabla\times\vtH=\vj\in L^2(\Omega)^{3}$, so $\vtH$ is in $H(\curl;\Omega)$ and satisfies equilibrium condition (\ref{eq:vtH}). Inequality (\ref{eq:errEst2}) then follows from Theorem~\ref{thm:estimator}.
\end{proof}

From Corollary \ref{cor:estimator}, it follows that a constant-free upper bound on the error can be obtained from any field $\vtHdel$ that satisfies \eqref{eq:vtHdel}.

An error estimator of this type was first introduced in \cite{braess08}, where it is referred to as an equilibrated residual error estimator. There, $\vj-\vj_h$ is decomposed into a sum of local divergence-free current distributions $\vjdel_i$ that have support on only a single vertex patch. The error estimator is then obtained by solving local problems of the form $\nabla\times\vtHdel_i=\vjdel_i$ for each vertex patch and by then taking the sum of all local fields $\vtHdel_i$. It is, however, not straightforward to decompose $\vj-\vj_h$ into local divergence-free current distributions. An explicit expression for $\vjdel_i$ is given in \cite{braess08} for the lowest-degree N\'ed\'elec element, but this expression cannot be readily extended to basis functions of arbitrary degree.

Here, we instead present an error estimator based on equilibration condition~\eqref{eq:vtHdel} that can be applied to elements of \emph{arbitrary} degree. Furthermore, instead of solving local problems on vertex patches, our estimator requires solving problems on only single elements, single faces, and small sets of nodes. The assumptions and a step-by-step derivation of the estimator are given in Sections~\ref{sec:assumptions} and \ref{sec:derivation} below, a brief summary is given in Section~\ref{sec:errEstOv}, and reliability and efficiency are proven in Section \ref{sec:eff}.

\subsection{Assumptions}
\label{sec:assumptions}
In order to compute the error estimator, we use polynomial function spaces of degree $k'\geq k$, where $k$ denotes the degree of the finite element approximation $\vu_h$, and assume that: 
\begin{itemize}
  \item[A1.] The magnetic permeability $\mu$ is piecewise constant. In particular, the domain $\Omega$ can be partitioned into a finite set of polyhedral subdomains $\{\Omega_i\}$ such that $\mu$ is constant on each subdomain $\Omega_i$. Furthermore, the mesh is assumed to be aligned with this partition so that $\mu$ is constant within each element. 
  \item[A2.] The current density $\vj$ is in $\RT_{k'}(\mathcal{T}_h)\cap H(\dvg^0;\Omega)$.
\end{itemize}
Although assumption A2 does not hold in general, we can always replace $\vj$ by a suitable projection $\pi_h\vj$ by taking, for example, $\pi_h$ as the standard Raviart--Thomas interpolation operator corresponding to the $\RT_{k'}(\mathcal{T}_h)$ space \cite{nedelec80}. The error is in that case bounded by
\begin{align*}
\|\mu^{1/2}(\vH-\vH_h)\|^2_{\Omega} &\leq \|\mu^{1/2}\vtHdel\|^2_{\Omega} + \|\mu^{1/2}(\vH-\vH') \|_{\Omega},
\end{align*}
where $\vH':=\mu^{-1}\nabla\times\vu'$ and where $\vu'$ is the solution to \eqref{eq:WF} with $\vj$ replaced by $\pi_h\vj$. If $\vj$ is sufficiently smooth, i.e. $\vj$ can be extended to a function $\vj^*\in H(\dvg;\mathbb{R}^3)\cap H^{k'}(\mathbb{R}^3)^3$ with compact support, then the term $\|\mu^{1/2}(\vH-\vH') \|_{\Omega}$ is of order $h^{k'+1}$; see Theorem \ref{thm:projectionError} in the appendix. This means that, if {$k'\geq k$}, then $\|\mu^{1/2}(\vH-\vH') \|_{\Omega}$ converges with a higher rate than $\| \mu^{1/2}(\vH-\vH_h) \|_{\Omega}$ and so we may assume that the term $\|\mu^{1/2}(\vH-\vH') \|_{\Omega}$ is negligible.

\subsection{Derivation of the error estimator}\label{sec:derivation}
\label{sec:errEstDer}
Before we derive the error estimator, we first write $\vj_h=\nabla\times\vH_h$ in terms of element and face distributions. For every $\vw\in \cC_0^{\infty}(\Omega)^3\cup\Nd_{k,0}(\Th)$, we can write
\begin{align*}
\langle \vj_h,\vw \rangle &= \langle \nabla\times \vH_h, \vw \rangle \\
&=  (\vH_h,\nabla\times\vw)_{\Omega} \\
&= \sum_{T\in\Th} \left[ (\nabla\times\vH_h,\vw)_T +  (\vH_h,\vn_T\times\vw)_{\partial T} \right] \\
&= \sum_{T\in\Th} \left[ (\nabla\times\vH_h,\vw)_T +  (\vH_h,\vn_T\times\vw)_{\partial T\setminus\partial\Omega} \right] \\
&= \sum_{T\in\Th} \left[ (\nabla\times\vH_h,\vw)_T +  (-\vn_T\times\vH_h,\vw)_{\partial T\setminus\partial\Omega} \right] \\
&= \sum_{T\in\Th} (\nabla\times\vH_h,\vw)_T + \sum_{f\in\Fhin} (-\jut{\vH_h},\vw)_f \\
&=: \sum_{T\in\Th} (\vj_{h,T},\vw)_T + \sum_{f\in\Fhin} (\vj_{h,f},\vw)_f,
\end{align*}
where $\langle\cdot,\cdot\rangle$ denotes the application of a distribution to a $\mathcal{C}_0^{\infty}(\Omega)^3$ function, $\Fhin$ denotes the set of all internal faces, $\vn_{T}$ denotes the normal unit vector to $\partial T$ pointing outward of $T$, and $\jut{\vH}|_f:=(\vn^+\times\vH^+ + \vn^-\times\vH^-)|_{f}$ denotes the tangential jump operator, with $\vn^{\pm}:=\vn_{T^{\pm}}$, $\vH^{\pm}:=\vH|_{T^{\pm}}$, and $T^+$ and $T^-$ the two adjacent elements of $f$. 

Since $\vu_h\in \Nd_{k,0}(\Th)$ and $\mu$ is piecewise constant, we have that $\vH_h|_T\in P_{k-1}(T)^3$. Therefore, $\vj_{h,T}\in P_{k-2}(T)^3\subset \RT_{k-1}(T)$ if $k\geq2$ and $\vj_{h,T}=0$ if $k=1$, and $\vj_{h,f}\in \{\vu\in P_{k-1}(f)^3 \;|\; \vn_f\cdot\vu=0 \}\subset \RT_{k}(f)$, where $\vn_f:=\vn^+|_f$ is a normal unit vector of $f$ and $\RT_{k}(f)$ is given by
\begin{align*}
&\qquad \RT_{k}(f) := \\
&\{\vu\in L^2(f) \;|\; \vu(\vx) = \vn_f\times(\vv(\vx) + \vx w(\vx)) \text{ for some }\vv\in P_{k-1}(f)^3, w\in P_{k-1}(f)\}.
\end{align*}
In other words, $\vj_h$ can be represented by $\RT_{k-1}(T)$ functions on the elements and $\RT_{k}(f)$ face distributions on the internal faces. 

We define $\vjdel:=\vj-\vj_h$ and can write
\begin{align}
\label{eq:vjdel}
\langle \vjdel, \vw\rangle  &=  \sum_{T\in\Th} (\vjdel_{T},\vw)_T +  \sum_{f\in\Fhin} (\vjdel_{f},\vw)_f &&\forall \vw\in \mathcal{C}_0^{\infty}(\Omega)^3 \cup\Nd_{k,0}(\Th),
\end{align}
where 
\begin{equation}
\label{eq:jTjf}
\vjdel_{T} := \vj|_T - \vj_{h,T} = \vj|_T -\nabla\times\vH_h|_T \quad \text{and}\quad
\vjdel_{f} := -\vj_{h,f} = \jut{\vH_h}|_f.
\end{equation}

We look for a solution of (\ref{eq:vtHdel}) of the form $\vtHdel=\vhHdel + \nabla_h\phi$, with $\vhHdel\in \Nd^{-1}_{k'}(\Th)$ and $\phi\in P^{-1}_{k'}(\Th)$ and where $\nabla_h$ denotes the element-wise gradient operator. The term $\vhHdel$ will take care of the element distributions of $\vjdel$ and the term $\nabla_h\phi$ will take care of the remaining face distributions.

In the following, we firstly describe how to compute $\vhHdel$ in Section~\ref{sec: vhHdel} and characterize the remainder $\vjdel - \nabla\times\vhHdel$ in Section~\ref{sec:remainder}. We then describe how to compute the jumps of $\phi$ on internal faces in Section~\ref{sec:jumpphi} and explain how to reconstruct $\phi$ from its jumps in Section~\ref{sec:phi}.

\subsubsection{Computation of $\vhHdel$}\label{sec: vhHdel}

We compute $\vhHdel$ by solving the local problems
\begin{subequations}
\label{eq:vhHdel}
\begin{align}
\nabla\times\vhHdel |_T &= \vjdel_{T} , &&\label{eq:vhHdel1a}\\
(\mu\vhHdel,\nabla\psi)_T &= 0 &&\forall \psi\in P_{k'}(T), \label{eq:vhHdel1b}
\end{align}
\end{subequations}
for each element $T\in\Th$. This problem is well-defined and has a unique solution due to the discrete exact sequence property 
\begin{align*}
P_{k'}(T)\xlongrightarrow{\text{$\nabla$}} \Nd_{k'}(T)
\xlongrightarrow{\text{$\nabla\times\ $}} \RT_{k'}(T)
\xlongrightarrow{\text{$\nabla\cdot\ $}} P_{k'-1}(T),
\end{align*}
and since $\nabla\cdot \vjdel_T =\nabla\cdot\vj|_T -  \nabla\cdot(\nabla\times\vH_h)|_T= 0$ and $\vjdel_T=\vj|_T -\nabla\times\vH_h|_T\in \RT_{k'}(T)$. This last property follows from the fact that $\vj|_T\in\RT_{k'}$ due to assumption A2 and $\vH_h|_T=\mu^{-1}\nabla\times\vu_h|_T\in P_{k-1}(T)^3\subset P_{k'-1}(T)$ due to assumption A1. 

\subsubsection{Representation of the remainder $\vjdel - \nabla\times\vhHdel$}
\label{sec:remainder}

Set
\begin{align*}
\vhjdel &:= \vjdel - \nabla\times\vhHdel.
\end{align*}
For every $\vw\in \mathcal{C}_0^{\infty}(\Omega)^3 \cup\Nd_{k,0}(\Th)$, we can write
\begin{align*}
\langle \vhjdel, \vw \rangle &= \langle \vjdel - \nabla\times\vhHdel, \vw\rangle  \\
&= \langle \vjdel, \vw \rangle - \langle \nabla\times\vhHdel,  \vw\rangle  \\
&= \langle \vjdel, \vw \rangle - (\vhHdel, \nabla\times \vw)_{\Omega}   \\
&\stackrel{\eqref{eq:vjdel}}{=} \sum_{T\in\Th} (\vjdel_T,  \vw)_T + \sum_{f\in\Fhin} (\vjdel_f,\vw)_f - \sum_{T\in\Th} \left[  (\vhHdel,\vn_T\times\vw)_{\partial T} +  (\nabla\times\vhHdel,  \vw)_T \right]  \\
&= \sum_{T\in\Th} \left[ (\vjdel_T - \nabla\times \vhHdel, \vw)_T -  (\vhHdel,\vn_T\times\vw)_{\partial T} \right] + \sum_{f\in\Fhin}  (\vjdel_f,\vw)_f  \\ 
&\stackrel{\eqref{eq:jTjf},\eqref{eq:vhHdel1a}}{=} \sum_{T\in\Th}  \left[ (\vzero, \vw)_T - (\vhHdel,\vn_T\times\vw)_{\partial  T\setminus\partial\Omega} \right] + \sum_{f\in\Fhin}  (\jut{\vH_h},\vw)_f \\
&= \sum_{T\in\Th} (\vn_T\times\vhHdel,\vw)_{\partial  T\setminus\partial\Omega} + \sum_{f\in\Fhin} (\jut{\vH_h},\vw)_f   \\
&= \sum_{f\in\Fhin} (\jut{\vhHdel},\vw)_{f} +  (\jut{\vH_h},\vw)_f   \\
&= \sum_{f\in\Fhin} (\jut{\vH_h + \vhHdel},\vw)_f, 
\end{align*}
so
\begin{align}
\label{eq:vhjdel}
\langle \vhjdel, \vw \rangle &= \sum_{f\in\Fhin} (\jut{\vH_h + \vhHdel},\vw)_f =: \sum_{f\in\Fhin} (\vhjdel_f,\vw)_f 
\end{align}%
for all $ \vw\in \mathcal{C}_0^{\infty}(\Omega)^3 \cup\Nd_{k,0}(\Th)$. This means that $\vhjdel$ can be represented by only face distributions, and since $\vH_h\in P^{-1}_{k-1}(\Th)^3\subset P^{-1}_{k'-1}(\Th)^3$ and $\vhHdel\in \Nd^{-1}_{k'}(\Th)$, we have that $\jut{\vH_h + \vhHdel}|_f \in \RT_{k'}(f)$ and therefore $\vhjdel_f\in \RT_{k'}(f)$.

\subsubsection{Computation of the jumps of $\phi$ on internal faces}
\label{sec:jumpphi}

It now remains to find a $\phi\in P^{-1}_{k'}(\Th)$ such that
\begin{align*}
\nabla\times \nabla_h\phi &= \vhjdel.
\end{align*}
For every $\vw\in \mathcal{C}_0^{\infty}(\Omega)^3 \cup\Nd_{k,0}(\Th)$, we can write
\begin{align*}
\langle \nabla\times\nabla_h\phi, \vw \rangle &= (\nabla_h\phi, \nabla\times\vw)_{\Omega} \\
&= \sum_{T\in\Th} \left[ (\nabla\times\nabla\phi,\vw)_T + (\nabla\phi,\vn_T\times\vw)_{\partial T} \right] \\
&= \sum_{T\in\Th} \left[ (\vzero,\vw)_T + (\nabla\phi,\vn_T\times\vw)_{\partial T\setminus\partial\Omega} \right] \\
&= \sum_{T\in\Th} - (\vn_T\times\nabla\phi,\vw)_{\partial T\setminus\partial\Omega} \\
&= \sum_{f\in\Fhin} (-\jut{ \nabla\phi},\vw)_f.
\end{align*}
Therefore, we need to find a $\phi\in P^{-1}_{k'}(\Th)$ such that
\begin{align*}
-\jut{ \nabla\phi}|_f &= \vhjdel_f &&\forall f\in\Fhin.
\end{align*}
To do this, we define, for each internal face $f\in\Fhin$, the scalar jump $\ju{\phi}_f:= (\phi^+ - \phi^-)|_{f}$ with $\phi^{\pm}:=\phi|_{T^{\pm}}$, two orthogonal unit tangent vectors $\vht_1$ and $\vht_2$ such that $\vht_1\times\vht_2=\vn_f:=\vn^+|_f$, differential operators $\partial_{t_i}:=\vht_i\cdot\nabla$, and the gradient operator restricted to the face: $\nabla_f := \vht_1\partial_{t_1} + \vht_2\partial_{t_2}$. We can then write
\begin{align*}
-\jut{ \nabla\phi}|_f &= -(\vn^+\times\nabla\phi^+ + \vn^-\times\nabla\phi^-)|_{f} \\
&=  -(\vn_f\times\nabla\phi^+ - \vn_f\times\nabla\phi^-)|_{f}  \\
&=  -(\vn_f\times\nabla_f\phi^+ - \vn_f\times\nabla_f\phi^-)|_{f}  \\
&= -\vn_f\times\nabla_f\ju{\phi}_f 
\end{align*}
for all $f\in\Fhin$. We therefore introduce an auxiliary variable $\lambda_f\in P_{k'}(f)$ and solve
\begin{subequations}
\label{eq:lambda1}
\begin{align}
-\vn_f\times\nabla_f\lambda_f &= \vhjdel_f, \label{eq:lambda1a} \\
(\lambda_f,1)_f&=0,  
\label{eq:lambda1b}
\end{align}
\end{subequations}
for each $f\in\Fhin$, where \eqref{eq:lambda1b} is only added to ensure a unique solution. In the next section, we will show the existence of and how to construct a $\phi\in P^{-1}_{k'}(\Th)$ such that $\ju{\phi}_f=\lambda_f$ for all $f\in\Fhin$. Now, we will prove that problem~\eqref{eq:lambda1} uniquely defines $\lambda_f$. We start by showing that~\eqref{eq:lambda1} corresponds to a 2D curl problem on a face. To see this, note that $\vn_f\times\nabla_f = \vht_2\partial_{t,1} - \vht_1\partial_{t,2}$. If we take the inner product of (\ref{eq:lambda1a}) with $\vht_1$ and $\vht_2$, we obtain
\begin{subequations}
\label{eq:lambda2}
\begin{align}
\partial_{t_2}\lambda_f &=  \jhdel_{f,t_1}, \label{eq:lambda2a} \\
-\partial_{t_1}\lambda_f&= \jhdel_{f,t_2}, \label{eq:lambda2b}
\end{align}
\end{subequations}
where $\jhdel_{f,t_i}:=\vht_i\cdot\vhjdel_f$, which is equivalent to a
2D curl problem on $f$. To show that~\eqref{eq:lambda1} is well-posed,
we use the discrete exact sequence in 2D:
\begin{align*}
\mathbb{R}\xlongrightarrow{\text{$\subset\ $}} P_{k'}(f)
\xlongrightarrow{\text{$\curl_f$}} \RT_{k'}(f)
\xlongrightarrow{\text{$\nabla\cdot\ $}} P_{k'-1}(f),
\end{align*}%
where $\curl_f:=(\partial_{t_2},-\partial_{t_1})$. Since $\vhjdel_f\in D_{k'}(f)$, it suffices to show that $\partial_{t_1}\jhdel_{f,t_1} + \partial_{t_2}\jhdel_{f,t_2}= \nabla_f\cdot \vhjdel_f = 0$. To prove this, we use that, for every $\psi\in \cC^\infty_0(\Omega)^3$,
\begin{align*}
\langle \vhjdel, \nabla\psi \rangle &= \langle \vjdel -\nabla\times\vhHdel, \nabla\psi \rangle \\
&= \langle \nabla\times(\vH-\vH_h-\vhHdel), \nabla\psi \rangle \\
&= (\vH-\vH_h-\vhHdel,\nabla\times\nabla\psi)_{\Omega} \\
&= (\vH-\vH_h-\vhHdel,\vzero)_{\Omega} \\
&= 0.
\end{align*}
Then, for every $\psi\in \mathcal{C}_0^{\infty}(\Omega)$, we can write
\begin{align*}
0&= \langle \vhjdel, \nabla\psi \rangle \\
&\stackrel{\eqref{eq:vhjdel}}{=} \sum_{f\in\Fhin} (\vhjdel_f,\nabla\psi)_f \\
&= \sum_{f\in\Fhin} (\vhjdel_f,\nabla_f\psi)_f  \\
&= \sum_{f\in\Fhin} \left[ (\vn_{\partial f}\cdot\vhjdel_f,\psi)_{\partial f} - (\nabla_f\cdot\vhjdel_f,\psi)_f \right]  \\
&= \sum_{f\in\Fhin} \left[ (\vn_{\partial f}\cdot\vhjdel_f,\psi)_{\partial f\setminus\partial\Omega} - (\nabla_f\cdot\vhjdel_f,\psi)_f \right]  \\
&= \sum_{f\in\Fhin} \sum_{e:e\subset\partial f\setminus\partial\Omega}  (\vn_{e,f}\cdot\vhjdel_f,\psi)_{e} - \sum_{f\in\Fhin} (\nabla_f\cdot\vhjdel_f,\psi)_f  \\
&=  \sum_{e\in\Ehin} \sum_{f:\,{\partial f}
\supset e} (\vn_{e,f}\cdot\vhjdel_f,\psi)_{e}  - \sum_{f\in\Fhin} (\nabla_f\cdot\vhjdel_f,\psi)_f \\
&=  \sum_{e\in\Ehin}  \left(\sum_{f:\,{\partial f}
\supset e} \vn_{e,f}\cdot\vhjdel_f,\psi\right)_{e}  - \sum_{f\in\Fhin} (\nabla_f\cdot\vhjdel_f,\psi)_f,
\end{align*}
where $\Ehin$ denotes the set of all internal edges, $\vn_{\partial  f}$ denotes the normal unit vector of $\partial f$ that lies in the same plane as $f$ and points outward of $f$, and $\vn_{e,f}:=\vn_{\partial f}|_e$. This implies that 
\begin{subequations}
\label{eq:prop_vhjdel1}
\begin{align}
\sum_{f:\partial f\supset e} \vn_{e,f}\cdot\vhjdel_f |_e &= 0 &&\forall e\in\Ehin,  \label{eq:prop_vhjdel1a} \\
\nabla_f\cdot\vhjdel_f &= 0 &&\forall f\in\Fhin,
\end{align}
\end{subequations}
so $\nabla_f\cdot\vhjdel_f=0$ for each internal face and, therefore, problem (\ref{eq:lambda1}) is well-defined and has a unique solution.

\subsubsection{Reconstruction of $\phi$ from its jumps on the internal faces}
\label{sec:phi}

After computing $\lambda_f$ for all internal faces, it remains to compute $\phi$ such that $\ju{\phi}_f=\lambda_f$ for all $f\in\Fhin$. To do this, we use standard Lagrangian basis functions. For each element $T$, let $\mathcal{Q}_T$ denote the set of nodes on element $\overline{T}$ for $P_{k'}(T)$. The barycentric coordinates of these nodes are given by \{$(\frac{i_1}{k'},\frac{i_2}{k'},\frac{i_3}{k'},\frac{i_4}{k'})\}_{i_1,i_2,i_3,i_4\geq 0, i_1+i_2+i_3+i_4=k'}$. Also, let $\Qh$ be the union of all element nodes. We then define the degrees of freedom for $\phi$, denoted by $\{\phi_{T,\vx}\}_{T\in\Th,\vx\in \mathcal{Q}_T}$, as the values of $\phi|_T$ at the nodes $\vx\in\mathcal{Q}_T$. Since, for each $f\in\Fhin$, the space $P_{k'}(f)$ is unisolvent on the nodes $\Qh\cap\overline{f}$ , we have that $\ju{\phi}_f=\lambda_f$ for all $f\in\Fhin$ if and only if
\begin{align*}
\phi_{T^+,\vx} - \phi_{T^-,\vx} &= \lambda_{f}(\vx) &&\forall f\in\Fhin, \, {\forall}\vx\in\Qh\cap\overline{f}.
\end{align*}
We can decouple this global problem into very local problems. In particular, for each node $\vx\in\Qh$, we can compute the small set of degrees of freedom $\{\phi_{T,\vx}\}_{T:\overline{T}\ni\vx}$ by solving
\begin{subequations}
\label{eq:phi1}
\begin{align}
\phi_{T^+,\vx} - \phi_{T^-,\vx} &= \lambda_{f}(\vx) &&\forall f\in\Fhin:\overline{f}\ni\vx, \label{eq:phi1a}\\
\sum_{T:\overline{T}\ni\vx} \phi_{T,\vx} &= 0. &&\label{eq:phi1b}
\end{align}
\end{subequations}
In this local problem, each degree of freedom corresponds to an element adjacent to $\vx$ and for any two degrees of freedom corresponding to two adjacent elements, the difference should be equal to $\lambda_f(\vx)$, with $f$ the face connecting the two elements. Condition~\eqref{eq:phi1b} is only added in order to ensure a unique local solution.

For a node $\vx$ in the interior of an element, there is only one overlapping element $T$ and the above results in $\phi_{T,\vx}=0$. The same applies to a node in the interior of a boundary face. For a node $\vx$ in the interior of an internal face $f$, there are only two adjacent elements $T^+$ and $T^-$ and the above results in $\phi_{T^+,\vx}=\frac12\lambda_f(\vx)$ and $\phi_{T^-,\vx}=-\frac12\lambda_f(\vx)$. For a node in the interior of an edge, the degrees of freedom correspond to the ring (for internal edges) or the partial ring (for boundary edges) of elements adjacent to that edge. Finally, for a node on a vertex, the degrees of freedom correspond to the cloud of elements adjacent to that vertex. 

For every cycle through elements adjacent to a node $\vx$, the corresponding differences should add up to zero. A cycle means a sequence of elements $T_1\rightarrow T_2\rightarrow \cdots\rightarrow T_{n+1}$ with $n\geq 3$, such that $T_1,T_2,\dots,T_n$ are all different from each other, $T_{n+1}=T_{1}$, and two consecutive elements are connected through a face. For a node in the interior of an internal edge, there is only one possible cycle, which is the cycle through the ring of elements adjacent to that edge. For a node on a vertex, the minimal cycles are the cycles around the internal edges adjacent to that vertex. Nodes in the interior of a face or in the interior of an element only have one or two adjacent elements and therefore have no cycles in their element patches. 

Therefore, in general, the minimal cycles for any node $\vx$ are the cycles around each internal edge connected to $\vx$. To prove that the overdetermined system~(\ref{eq:phi1}) is well-posed, it is therefore sufficient to check if, for each internal edge, the differences corresponding to the cycle around that edge sum to zero.

We can write this condition more formally. Let $e\in\Ehin$ be an internal edge, let $\vht_e$ be a tangent unit vector of $e$, and let $T_1\rightarrow T_2\rightarrow \cdots\rightarrow T_{n+1}$ be a cycle rotating counter-clockwise when looking towards $\vht_e$. This means that the normal unit vector of the face $T_i\cap T_{i+1}$ pointing out of $T_{i+1}$ is given by $\vn_{T_{i+1}}|_f=\vht_e\times \vn_{e,f}=:\vn_{f,e}$ (recall that $\vn_{e,f}=\vn_{\partial f}|_e$). Also, let $\psi\in P^{-1}_{k'}(\Th)$. The sum of the differences of $\psi$ for this cycle can be written as
\begin{align*}
\sum_{i=1}^n (\psi_{T_{i+1}}-\psi_{T_i})|_e = (\psi_{T_{n+1}}-\psi_{T_1})|_e = (\psi_{T_{1}}-\psi_{T_1})|_e=  0.
\end{align*}
Now, let $f=T_i\cap T_{i+1}$ and note that 
\begin{align*}
(\psi_{T_{i+1}}-\psi_{T_i})|_f &= \begin{cases}
\ju{\psi}_f ,& \text{if\ } \vn_f=\vn_{T_{i+1}}|_f, \\
-\ju{\psi}_f ,& \text{if\ } \vn_f=-\vn_{T_{i+1}}|_f.
\end{cases}
\end{align*}
Therefore, we can write 
\begin{align*}
(\psi_{T_{i+1}}-\psi_{T_i})|_f &= \big(\vn_f\cdot(\vn_{T_{i+1}}|_f)\big)\ju{\psi}_f = (\vn_f\cdot\vn_{f,e})\ju{\psi}_f.
\end{align*}
The sum of the differences of the cycle around $e$ can therefore be rewritten as
\begin{align*}
0=\sum_{i=1}^n (\psi_{T_{i+1}}-\psi_{T_i})|_e = \sum_{f:\partial f\supset e} (\vn_f\cdot\vn_{f,e})\ju{\psi}_f|_e
\end{align*}
and from this, we obtain the conditions
\begin{align}
\label{eq:cond_lambda1}
r_e:=\sum_{f:\partial f\supset e} (\vn_{f}\cdot\vn_{f,e})\lambda_{f}|_e &= 0 &&\forall e\in\Ehin.
\end{align}
Problem (\ref{eq:phi1}) is therefore well-posed provided that (\ref{eq:cond_lambda1}) is satisfied.

\begin{figure}[h]
\includegraphics[width=0.5\textwidth]{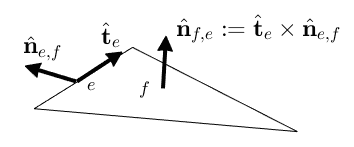}
\caption{Illustration of $\vht_e$, $\vn_{e,f}$, and $\vn_{f,e}$. The vector $\vn_{f}$ either equals $\vn_{f,e}$ or $-\vn_{f,e}$.}
\label{fig:triangle}
\end{figure}

To prove (\ref{eq:cond_lambda1}), we first prove that, for each $e\in\Ehin$, $r_e$ is constant and then prove that, for each $e\in\Ehin$, $(r_e,1)_e=0$.

Let $e$ be an edge and let $f$ be an adjacent face, and consider (\ref{eq:lambda2}) with $\vht_1=\vht_e$ and $\vht_2=(\vn_f\cdot\vn_{f,e})\vn_{e,f}$. An illustration of $\vht_e$, $\vn_{e,f}$, and $\vn_{f,e}$ is given in Figure \ref{fig:triangle}. The condition $\vn_f=\vht_1\times\vht_2$ is still satisfied, since either $\vn_f=\vn_{f,e}$ or $\vn_f=-\vn_{f,e}$ and so
\begin{align*}
\vht_1\times\vht_2 = \vht_e\times\vn_{e,f}(\vn_f\cdot\vn_{f,e}) = \vn_{f,e}(\vn_f\cdot\vn_{f,e})= \vn_f.
\end{align*}
It then follows from (\ref{eq:lambda2b}) that
\begin{align*}
-\partial_{t_e}\lambda_f|_e = -\partial_{t_1}\lambda_f|_e = \jhdel_{f,t_2}|_e = (\vn_f\cdot\vn_{f,e})\vn_{e,f}\cdot\vhjdel_f|_e,
\end{align*}
where $\partial_{t_e}:=\vht_e\cdot\nabla$. Multiplying the above by $-(\vn_f\cdot\vn_{f,e})$ and summing over all faces adjacent to $e$ results in
\begin{align*}
\sum_{f:\partial f\supset e} (\vn_f\cdot\vn_{f,e}) \partial_{t_e} \lambda_{f}|_e &= -\sum_{f:\partial f\supset e} (\vn_f\cdot\vn_{f,e})^2\vn_{e,f}\cdot\vhjdel_f|_e \\
&= -\sum_{f:\partial f\supset e} \vn_{e,f}\cdot\vhjdel_f|_e \\
&= 0,
\end{align*}
where the last line follows from (\ref{eq:prop_vhjdel1a}). We thus have
\begin{align}
\label{eq:prop_lambda1}
\partial_{t_e}r_e &= \sum_{f:\partial f\supset e} (\vn_f\cdot\vn_{f,e}) \partial_{t_e} \lambda_{f}|_e =0 &&\forall e\in\Ehin.
\end{align}
This implies that $r_e$ is a constant. To prove (\ref{eq:cond_lambda1}), it therefore remains to show that $(r_e,1)_e=0$ for all internal edges. 

To prove this, we define $\vtht_e$ to be the lowest-order N\'ed\'elec basis function corresponding to an internal edge $e$ and scaled such that $(\vht_e\cdot\vtht_e)|_e =1$ and derive
\begin{align*}
\langle \vhjdel, \vtht_e\rangle &= \langle \vjdel - \nabla\times\vhHdel, \vtht_e\rangle \\
&= \langle\nabla\times(\vH-\vH_h-\vhHdel), \vtht_e\rangle_{\Omega}  \\
&= (\vH-\vH_h-\vhHdel,\nabla\times\vtht_e)_{\Omega}  \\
&= (\mu^{-1}\nabla\times\vu, \nabla\times\vtht_e)_{\Omega} - (\mu^{-1}\nabla\times\vu_h, \nabla\times\vtht_e)_{\Omega}  - (\vhHdel,\nabla\times\vtht_e)_{\Omega} \\
&\stackrel{\eqref{eq:vhHdel1b}}{=} (\vj,\nabla\times\vtht_e)_{\Omega} - (\vj,\nabla\times\vtht_e)_{\Omega} - 0 \\
&= 0,
\end{align*}
where the first two terms in the fifth line follow from the definition of $\vu$ and $\vu_h$ and the last term in the fifth line follows from (\ref{eq:vhHdel1b}), assumption A1, and the fact that $\nabla\times\vtht_e$ is piecewise constant. We can then derive
\begin{align*}
0 =\langle \vhjdel, \vtht_e\rangle &\stackrel{\eqref{eq:vhjdel}}{=} \sum_{f\in\Fhin} (\vhjdel_f,\vtht_e)_f \\
&= \sum_{f:\partial f\supset e} (\vhjdel_f,\vtht_e)_f \\
& \stackrel{\eqref{eq:lambda1a}}{=} \sum_{f:\partial f\supset e} (-\vn_f\times\nabla_f\lambda_f,\vtht_e)_f \\
&= \sum_{f:\partial f\supset e} (\nabla_f\lambda_f,\vn_f\times\vtht_e)_f \\
&= \sum_{f:\partial f\supset e} \left[ (\vn_{\partial f}\lambda_f,\vn_f\times\vtht_e)_{\partial f} - \big(\lambda_f,\nabla_f\cdot(\vn_f\times\vtht_e)\big)_f \right] \\
&= \sum_{f:\partial f\supset e} \left[ (-\vn_f\times\vn_{\partial f}\lambda_f,\vtht_e)_{\partial f} + (\lambda_f,\vn_f\cdot(\nabla_f\times\vtht_e))_f \right] \\
& \stackrel{\eqref{eq:lambda1b}}{=} \sum_{f:\partial f\supset e}(-\vn_f\times\vn_{\partial f}\lambda_f,\vtht_e)_{\partial f} + 0, 
\end{align*}
where the second line follows from the fact that the tangent components of $\vtht_e$ are zero on all faces that are not adjacent to $e$, and the last line follows from (\ref{eq:lambda1b}) and the fact that $\nabla_f\times\vtht_e$ is constant on $f$. We continue to obtain
\begin{align*}
0 &= \sum_{f:\partial f\supset e}(-\vn_f\times\vn_{\partial f}\lambda_f,\vtht_e)_{\partial f} \\
&= \sum_{f:\partial f\supset e}  \sum_{\tilde{e}:\tilde{e}\subset\partial{f}} (-\vn_f\times\vn_{\tilde{e},f} \lambda_f,\vtht_e)_{\tilde{e}} \\
&= \sum_{f:\partial f\supset e}  \sum_{\tilde{e}:\tilde{e}\subset\partial{f}} \big(-(\vn_f\cdot\vn_{f,\tilde e})\vn_{f,\tilde e}\times\vn_{\tilde e,f} \lambda_f,\vtht_e \big)_{\tilde{e}} \\
&= \sum_{f:\partial f\supset e}  \sum_{\tilde{e}:\tilde{e}\subset\partial{f}} \big(-(\vn_f\cdot\vn_{f,\tilde e})(\vht_{\tilde e}\times\vn_{\tilde e,f})\times\vn_{\tilde e,f} \lambda_f,\vtht_e \big)_{\tilde{e}} \\
&= \sum_{f:\partial f\supset e}\sum_{\tilde{e}:\tilde{e}\subset\partial{f}} \big( (\vn_f\cdot\vn_{f,\tilde e}) \vht_{\tilde{e}}\lambda_f,\vtht_e \big)_{\tilde{e}}  \\
&=\sum_{f:\partial f\supset e}\sum_{\tilde{e}:\tilde{e}\subset\partial{f}} \big((\vn_f\cdot\vn_{f,\tilde e})\lambda_f,\vht_{\tilde{e}}\cdot\vtht_e \big)_{\tilde{e}} \\
&=\sum_{f:\partial f\supset e} \big((\vn_f\cdot\vn_{f,e})\lambda_f,1 \big)_{e},
\end{align*}
where the third line follows from the fact that either $\vn_f=\vn_{f,\tilde e}$ or $\vn_f=-\vn_{f,\tilde e}$, the fifth line follows from the property $(\vct{a}\times\vct{b})\times\vct{c}=-\vct{a}(\vct{b}\cdot\vct{c})+\vct{b}(\vct{a}\cdot\vct{c})$ and the fact that $\vht_{\tilde e}$ and $\vn_{\tilde e,f}$ are orthogonal, and the last line follows from the fact that $(\vht_{\tilde e}\cdot\vtht_e)|_{\tilde e}=0$ for all edges $\tilde e\neq e$ and from $(\vht_e\cdot\vtht_e)|_e=1$. We then continue to obtain
\begin{align*}
0 &= \sum_{f:\partial f\supset e} \big( (\vn_f\cdot\vn_{f,e}) \lambda_f,1 \big)_{e} \\
&= \left(\sum_{f:\partial f\supset e} (\vn_f\cdot\vn_{f,e}) \lambda_f, 1 \right)_{e} \\
&= (r_e,1)_{e}.
\end{align*}
Therefore, $(r_e,1)_{e}=0$ and hence $r_e=0$ for all internal edges and so problem (\ref{eq:phi1}) is well-posed.

\subsection{Summary of computing the error estimator}
\label{sec:errEstOv}
We fix $k'\geq k$ and compute our error estimator in four steps.

\emph{Step 1.} We compute $\vhHdel\in \Nd^{-1}_{k'}(\Th)$ from the datum $\vj$ and the numerical solution~$\vH_h$ by solving
\begin{subequations}
\label{eq:vhHdel3}
\begin{align}
\nabla\times\vhHdel |_T &= \vjdel_T = \vj|_T-\nabla\times\vH_h|_T, && \\
(\mu\vhHdel,\nabla\psi)_T &= 0 &&\forall \psi\in P_{k'}(T), \label{eq:vhHdel3b}
\end{align}
\end{subequations}
for each $T\in\Th$. 

\emph{Step 2.} For each internal face $f\in\Fhin$, let $T^+$ and $T^-$ denote the two adjacent elements, let $\vn^{\pm}$ denote the normal unit vector pointing outward of $T^{\pm}$, let $\vH^{\pm}:=\vH|_{T^{\pm}}$ denote the vector field restricted to $T^{\pm}$, let $\jut{\vH}|_f := (\vn^+\times\vH^+ + \vn^-\times\vH^-)|_{f}$ denote the tangential jump operator, and let $\nabla_f$ denote the gradient operator restricted to face $f$. We set $\vn_f:=\vn^+|_f$ and compute $\lambda_f\in P_{k'}(f)$ by solving
\begin{subequations}
\label{eq:lambda3}
\begin{align}
-\vn_f\times\nabla_f\lambda_f &= \vhjdel_f = \vjdel_f+\jut{\vhHdel}|_f =  \jut{\vH_h + \vhHdel}|_f , \\
(\lambda_f,1)_f &=0, 
\end{align}
\end{subequations}
for each internal face $f\in\Fhin$. 

\emph{Step 3.} We compute $\phi\in P^{-1}_{k'}(\Th)$ by solving, for each $\vx\in\Qh$, the small set of degrees of freedom $\{\phi_{T,\vx}\}_{T:\overline{T}\ni\vx}$ such that
\begin{subequations}
\label{eq:phi3}
\begin{align}
\phi_{T^+,\vx} - \phi_{T^-,\vx} &= \lambda_{f}(\vx) &&\forall f\in\Fhin:\partial f\ni\vx, \\
\sum_{T:\overline{T}\ni\vx} \phi_{T,\vx} &= 0, &&
\end{align}
\end{subequations}
where $\Qh$ denotes the set of standard Lagrangian nodes corresponding to the finite element space $P_{k'}(\Th)$ and $\phi_{T,\vx}$ denotes the value of $\phi|_T$ at node $\vx$. 

\emph{Step 4.} We compute the field
\begin{align*}
\vtHdel = \vhHdel + \nabla_h\phi,
\end{align*}
where $\nabla_h$ denotes the element-wise gradient operator, and compute the error estimator
\begin{align}
\label{eq:eta}
\eta_h :=\|\mu^{1/2}\vtHdel\|_{\Omega} =  \left(\sum_{T\in \Th}\eta_T^2\right)^{1/2}, &\quad\text{where}\quad
\eta_T :=\|\mu^{1/2}\vtHdel\|_T.
\end{align}

\begin{rem}
For the case $k=1$, this algorithm requires solving local problems that involve 6 unknowns per element (Step 1), 3 unknowns per face (Step 2), and $(\# T:\overline{T}\ni\nu) \approx 24$ unknowns per vertex $\nu$ (Step 3). Hence, the third step is expected to be the most computationally expensive step, although it still involves only 3 times as few unknowns as the vertex-patch problems of \cite{braess08}.
\end{rem}

\subsection{Reliability and Efficiency}
\label{sec:eff}

The results of the previous sections immediately give the following theorem.

\begin{thm}[reliability]
Let $\vu$ be the solution to (\ref{eq:WF}), let $\vu_h$ be the solution to (\ref{eq:FEM}), and set $\vH:=\mu^{-1}\nabla\times\vu$ and $\vH_h:=\mu^{-1}\nabla\times\vu_h$. Also, fix $k'\geq k$ and assume that assumptions A1 and A2 hold true. Then the problems in Steps~1--3 are all well-defined and have a unique solution. Furthermore, if $\vtHdel$ is computed by following Steps~1--4 in Section \ref{sec:errEstOv}, then
\begin{align}
\label{eq:errEst3}
\|\mu^{1/2}(\vH-\vH_h)\|_{\Omega} &\leq \|\mu^{1/2}\vtHdel \|_{\Omega}=\eta_h.
\end{align}
\end{thm}

We also prove local efficiency of the estimator and state it as the following theorem.

\begin{thm}[local efficiency]
\label{th:efficiency}
Let $\vu$ be the solution to (\ref{eq:WF}), let $\vu_h$ be the solution to (\ref{eq:FEM}), and set $\vH:=\mu^{-1}\nabla\times\vu$ and $\vH_h:=\mu^{-1}\nabla\times\vu_h$. Also, fix $k'\geq k$ and assume that assumptions A1 and A2 hold true. If $\vtHdel$ is computed by following Steps~1--4 in Section \ref{sec:errEstOv}, then
\begin{align}
\label{eq:efficiency}
\eta_T &= \|\mu^{1/2}\vtHdel\|_T \leq C\sum_{T':\overline{T'}\cap\overline{T}\neq\emptyset} \|\mu^{1/2}(\vH-\vH_h)\|_{T'}
\end{align}
for all $T\in\Th$, where $C$ is some positive constant that depends on the magnetic permeability $\mu$, the shape-regularity of the mesh and the polynomial degree $k'$, but not on the mesh width $h$.
\end{thm}

\begin{proof}
In this proof, we always let $C$ denote some positive constant that does not depend on the mesh width $h$, but may depend on the magnetic permeability $\mu$, the shape-regularity of the mesh and the polynomial degree $k'$. 

\emph{1.} Fix $T\in\Th$ and let $\vhHdel_T\in\Nd_{k'}(T)$ be the solution to \eqref{eq:vhHdel3}. To obtain an upper bound for $\|\mu^{1/2}\vhHdel_T\|_T$, we consider (\ref{eq:vhHdel3}) for a reference element. Let $\hat{T}$ denote the reference tetrahedron, let $\boldsymbol{\varphi}_T:\hat{T}\rightarrow T$ denote the affine element mapping, and let $\jac_T := [\frac{\partial \boldsymbol{\varphi}_T}{\partial\hat{x}_1}\, \frac{\partial \boldsymbol{\varphi}_T}{\partial\hat{x}_2}\, \frac{\partial \boldsymbol{\varphi}_T}{\partial\hat{x}_3}]$ be the Jacobian of $\boldsymbol{\varphi}_T$. Also, let $\vhGdel_{\hat T}\in\Nd_{k'}(\hat T)$ be the solution of
\begin{subequations}
\label{eq:vhGdel1}
\begin{align}
\nabla\times\vhGdel_{\hat T} &= \vjdel_{\hat T} &&\text{in }\hat{T} 
\label{eq:vhGdel1a} \\
(\vhGdel_{\hat T},\nabla\psi)_{\hat T} &= 0 &&\forall \psi\in P_{k'}(\hat T),
\end{align}
\end{subequations}
where $\vjdel_{\hat T}$ is defined as the pull-back of $\vjdel_{T}$ through the Piola transformation, namely such that $\vjdel_T\circ \boldsymbol{\varphi}_T= \frac{1}{\mathrm{det}(\jac_T)}\jac_T \vjdel_{\hat T}$, with $\mathrm{det}(\jac_T)$ the determinant of $\jac_T$. From the discrete Friedrichs inequality, it follows that
\begin{align}
\label{eq:eff1}
\| \vhGdel_{\hat T}\|_{\hat T} \leq C \| \vjdel_{\hat T} \|_{\hat T}.
\end{align}
Furthermore, if we set $\vhGdel_T$ as the push-forward of $\vhGdel_{\hat T}$ through the covariant transformation, namely
$\vhGdel_T\circ \boldsymbol{\varphi}_T:=\jac_T^{-t} \vhGdel_{\hat T}$, 
where $\jac_T^{-t}$ is the inverse of the transpose of the Jacobian $\jac_T$, then it can be checked that
\begin{align*}
\nabla\times\vhGdel_T &= \vjdel_T &&\text{in }T.
\end{align*}
From \eqref{eq:eff1}, it also follows that
\begin{align}
\label{eq:eff2}
\|\mu^{1/2}\vhGdel_T\|_T &\leq Ch_T \|\vjdel_T\|_T,
\end{align}
where $h_T$ denotes the diameter of $T$. Now, since $\vhGdel_T,\vhHdel_T\in\Nd_{k'}(T)$ and since $\nabla\times(\vhHdel_T-\vhGdel_T)=\vjdel_T-\vjdel_T=\vct{0}$, we have that $\vhHdel_T-\vhGdel_T = \nabla\psi$ for some $\psi\in P_{k'}(T)$. From \eqref{eq:vhHdel3b} and Pythagoras's theorem, it then follows that
\begin{align*}
\| \mu^{1/2} \vhGdel_T\|_T^2 &= \| \mu^{1/2} \vhHdel_T\|_T^2 + \| \mu^{1/2} (\vhGdel_T-\vhHdel_T)\|_T^2\geq \|\mu^{1/2} \vhHdel_T\|_T
\end{align*}
and from \eqref{eq:eff2}, it then follows that
\begin{align}
\label{eq:eff3}
\| \mu^{1/2} \vhHdel_T\|_T &\leq Ch_T\|\vjdel_T\|_T.
\end{align}

\emph{2.} Fix $f\in\Fhin$ and let $\lambda_f$ be the solution to
\eqref{eq:lambda3}. In a way, similarly as for $\vhHdel_T$, we can obtain the bound
\begin{align}
\label{eq:eff5}
\|\lambda_f\|_f &\leq Ch_f\|\vhjdel_f\|_f,
\end{align}
where $h_f$ denotes the diameter of the face $f$. Since $\vhjdel_f=\vjdel_f+\jut{\vhHdel}|_f$, we can use~\eqref{eq:eff5}, the triangle inequality, the trace inequality, discrete inverse inequality, and \eqref{eq:eff3} to obtain
\begin{align}
\|\lambda_f\|_f &\leq Ch_f\left(\|\vjdel_f\|_f + \|\jut{\vhHdel}\|_f \right) \nonumber \\
&\leq Ch_f\left(\|\vjdel_f\|_f + h_{T^+}^{-1/2}\|\vhHdel_{T^+}\|_{T^+} + h_{T^-}^{-1/2}\|\vhHdel_{T^-}\|_{T^-} \right) \nonumber \\
&\leq C\left(h_f\|\vjdel_f\|_f + h_{T^+}^{3/2}\|\vjdel_{T^+}\|_{T^+} + h_{T^-}^{3/2}\|\vjdel_{T^-}\|_{T^-} \right), \label{eq:eff6}
\end{align}
where $T^+$ and $T^-$ are the two adjacent elements of $f$.

\emph{3.} Fix $\vx\in\Qh$ and let $\{\phi_{T,\vx}\}_{T:\overline{T}\ni\vx}$ be the solution to \eqref{eq:phi3}. Since $\{\phi_{T,\vx}\}_{T:\overline{T}\ni\vx}$ depends linearly on $\{\lambda_f(\vx)\}_{f:\overline{f}\ni\vx}$, and since the number of possible nodal patches, which depends on the mesh regularity, is finite, we can obtain the bound
\begin{align*}
\left(\sum_{T:\overline{T}\ni\vx} \phi_{T,\vx}^2\right)^{1/2} \leq C\left(\sum_{f:\overline{f}\ni\vx} \lambda_{f}(\vx)^2  \right)^{1/2}.
\end{align*}
Now, fix $T\in\Th$. From the above and \eqref{eq:eff6}, we can obtain
\begin{align*}
\|\phi\|_T &\leq C \sum_{f:\overline{f}\cap\overline{T}\neq\emptyset} h_f^{1/2}\|\lambda_f\|_f \\
&\leq C \sum_{f:\overline{f}\cap\overline{T}\neq\emptyset} \left[h_f^{3/2}\|\vjdel_f\|_f + h_{T^+}^{2}\|\vjdel_{T^+}\|_{T^+} + h_{T^-}^{2}\|\vjdel_{T^-}\|_{T^-} \right].
\end{align*}
Note that $h_f,h_{T'}\leq Ch_T$ for all $f:\overline{f}\cap T\neq\emptyset$ and $T':\overline{T'}\cap \overline{T}\neq\emptyset$, due to the regularity of the mesh, which is incorporated in the constant $C$. Using the above and the discrete inverse inequality, we then obtain
\begin{align}
\|\mu^{1/2}\nabla\phi\|_T &\leq Ch_T^{-1}\|\phi\|_T \nonumber \\
&\leq C\sum_{f:\overline{f}\cap\overline{T}\neq\emptyset} \left[h_f^{1/2}\|\vjdel_f\|_f + h_{T^+}\|\vjdel_{T^+}\|_{T^+} + h_{T^-}\|\vjdel_{T^-}\|_{T^-} \right]. \label{eq:eff7}
\end{align}

\emph{4.} We now use the efficiency estimate of the residual error estimator established in~\cite{beck00}. This estimate can be written as
\begin{align*}
h_T\|\vjdel_T\|_T &= h_T\|\vj - \nabla\times\vH_h \|_T \leq C\|\mu^{1/2}(\vH-\vH_h)\|_T \\
h_f^{1/2}\|\vjdel_f\|_f &= h_f^{1/2}\| \jut{\vH_h}\|_f \leq C\left(\|\mu^{1/2}(\vH-\vH_h)\|_{T^+}  + \|\mu^{1/2}(\vH-\vH_h)\|_{T^-}\right)
\end{align*}
for all $T\in\Th$, $f\in\Fhin$. Using that $\vtHdel|_T = \vhHdel_T + \nabla\phi|_T$, the triangle inequality, \eqref{eq:eff3}, \eqref{eq:eff7}, and the above, we obtain
\begin{align*}
\|\mu^{1/2}\vtHdel\|_T &\leq \|\mu^{1/2}\vhHdel\|_{T} + \|\mu^{1/2}\nabla\phi \|_T \\
&\leq Ch_T\|\vjdel\|_T + C\sum_{f:\overline{f}\cap\overline{T}\neq\emptyset} \left[h_f^{1/2}\|\vjdel_f\|_f + h_{T^+}\|\vjdel_{T^+}\|_{T^+} + h_{T^-}\|\vjdel_{T^-}\|_{T^-} \right] \\
&\leq C\!\!\!\sum_{T':\overline{T'}\cap\overline{T}\neq\emptyset} \|\mu^{1/2}(\vH-\vH_h)\|_{T'},
\end{align*}
which completes the proof.
\end{proof}

\begin{rem}
\label{rem:p-robustness}
In Theorem~\ref{th:efficiency}, we proved that the constant in the efficiency estimate is independent of the mesh resolution, but may depend on $k^\prime$, as we used discrete inverse inequalities and the efficiency result stated in~\cite{beck00}. Whether it is also independent of $k^\prime$ is an open issue.  
\end{rem}

\section{Numerical experiments}
\label{sec:numerics}
In this section, we present several numerical results for the unit cube and the L-brick domain with constant magnetic permeability $\mu=1$ for the \textit{a posteriori} estimator constructed according to Steps 1--4 in Section~\ref{sec:errEstOv}.

For efficiency of the computations, we choose the same polynomial degree for the computation of the \textit{a posteriori} error estimator as for the approximation $\vu_h$, i.e. $k'=k$. In the numerical experiments, we do not project the right hand side $\vj$ onto $\RT_k(\Th)$, but solve the local problems of Step~1 in variational form. This introduces small compatibility errors into Step~2 and Step~3, which we observe to be negligible for the computation of the error estimator, cf. also the previous discussion on assumption A2 in Section~\ref{sec:assumptions}. The orthogonality conditions of the local problems of Step~1 and Step~2 are incorporated via Lagrange multipliers and the resulting local saddle point problems are solved with a direct solver. For the computation of $\phi$ in Step~3, we solve the local overdetermined systems with the least-squares method. Since we have shown that the solutions to those discrete problems are unique, the least-squares method computes those discrete unique solutions.

Instead of solving the full discrete problem~\eqref{eq:FEM}, the
numerical approximation $\vu_h$ is computed by solving the singular
system corresponding to~\eqref{eq:FEMa} only, since the gauge
condition \eqref{eq:FEMb} does not affect the variable of interest
$\vH_h:=\mu^{-1}\nabla\times\vu$. In order to do this, we use the
preconditioned conjugate gradient algorithm in combination with a
multigrid preconditioner \cite{AFW2000,H1999}. To ensure that, in the
presence of quadrature errors, the discretised right-hand side remains
discretely divergence free, a small gradient correction is added
following \cite[Section 4.1]{creuse19a}. Our implementation of
$\Nd_k(\Th)$ is based on the hierarchical basis functions from~\cite{ZaglmayrPhD}.

We evaluate the reliability and efficiency of the \textit{a posteriori} error estimator $\eta_h$ defined in~\eqref{eq:eta} for uniform and for adaptive meshes. For adaptive mesh refinement, we employ the standard adaptive finite element algorithm. Firstly we \emph{solve} the discrete problem~\eqref{eq:FEM}, then we compute the \textit{a posteriori} error estimator to \emph{estimate} the error; based on the local values $\eta_T$ of the estimator, we \emph{mark} elements for refinement based on the bulk marking strategy \cite{Doerfler1996} with bulk parameter $\theta=0.5$, and finally we \emph{refine} the marked elements based on a bisection strategy~\cite{AMP2000}.

\subsection{Unit cube example}
\begin{figure}[t]
\includegraphics[width=0.49\textwidth]{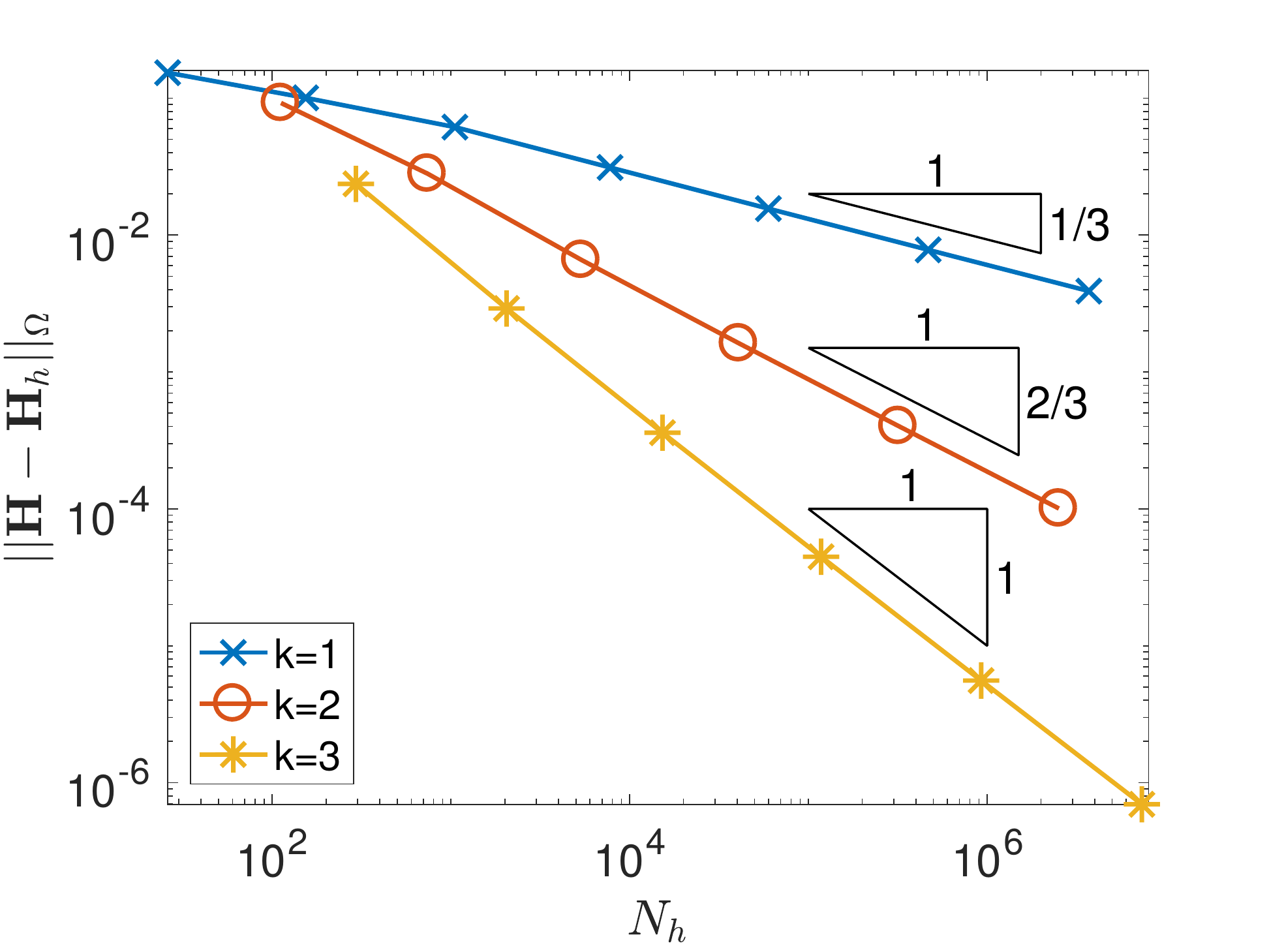}
\includegraphics[width=0.49\textwidth]{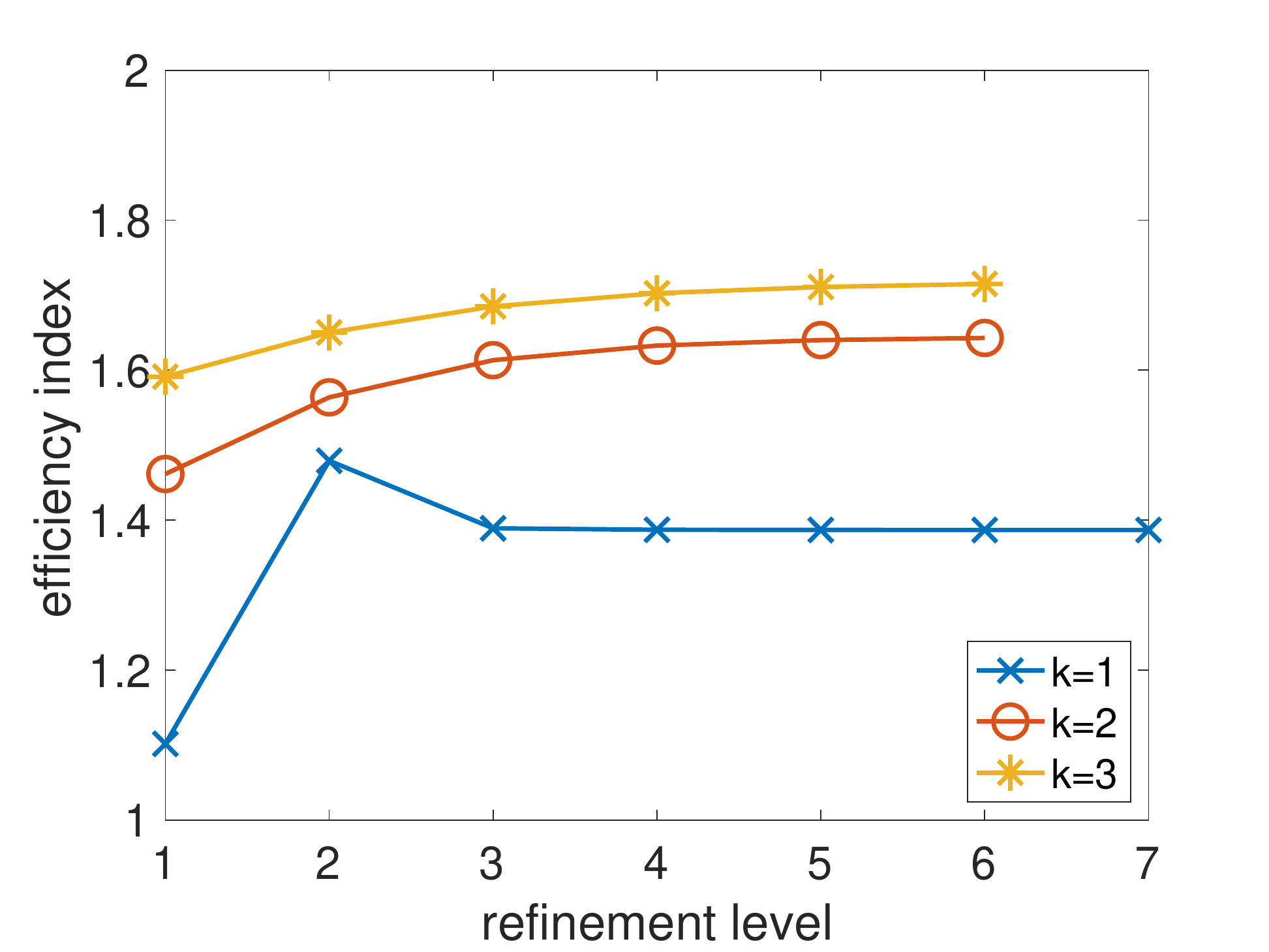}
\caption{Error and efficiency indices for the unit cube example with polynomial solution and uniformly refined meshes.}
\label{fig:unit:cube}
\end{figure}
For the first example we chose the unit cube $\Omega=(0,1)^3$ with homogeneous boundary conditions and right hand side $\vj$ according to the polynomial solution 
\begin{align*}
\vu(x,y,z) = \left( \begin{array}{c} y(1-y)z(1-z)\\ x(1-x)z(1-z) \\x(1-x)y(1-y) \end{array}\right).
\end{align*}

In Figure~\ref{fig:unit:cube}, we present the errors $\|\vH-\vH_h \|_{\Omega}$ and efficiency indices $\eta_h/\| \vH-\vH_h \|_{\Omega}$ for $k=1,2,3$, on a sequence of uniformly refined meshes. We observe that the error converges with optimal rates $\mathcal{O}(h^k) = \mathcal{O}(N_h^{-k/3})$, $N_h = \operatorname{dim}(\Nd_k(\Th))$, and that the error estimator is reliable and efficient with efficiency constants between $1$ and $2$.

Note that, for $k=3$, $\vj$ belongs to $\RT_k(\Th)$, hence in that case assumption A2 is valid. We observe that, in the other cases $k=1,2$, the estimator is reliable and efficient as well, even though $\vj$ does not belong to $\RT_k(\Th)$. Thus the error introduced in the computation of the error estimator by not satisfying A2 is negligible.

\subsection{L-brick example}
\begin{figure}[t]
\includegraphics[width=0.49\textwidth]{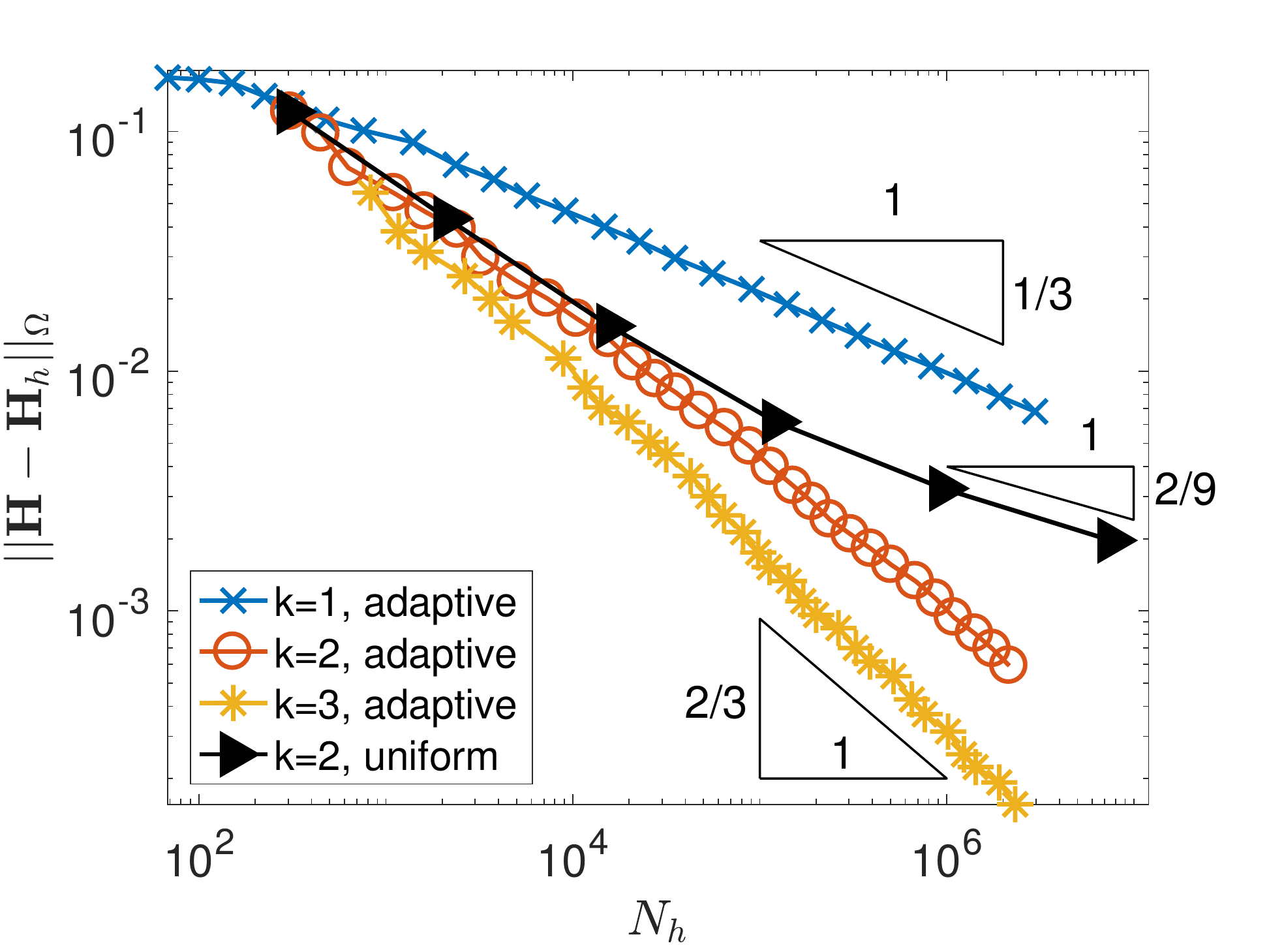}
\includegraphics[width=0.49\textwidth]{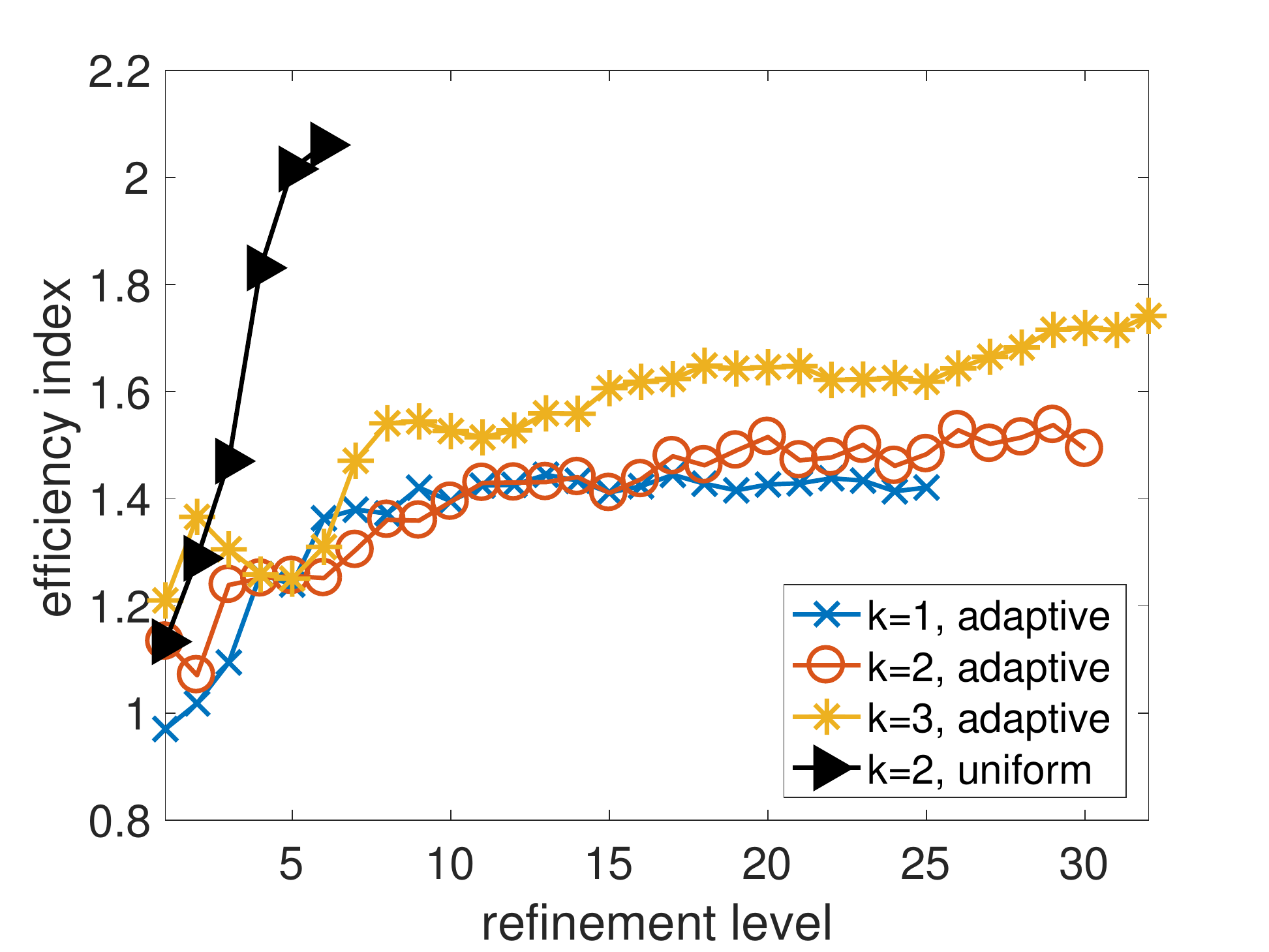}
\caption{Error and efficiency indices for adaptive mesh refinement for the L-brick example.}
\label{fig:L-brick}
\end{figure}
\begin{figure}[t]
\includegraphics[width=0.49\textwidth]{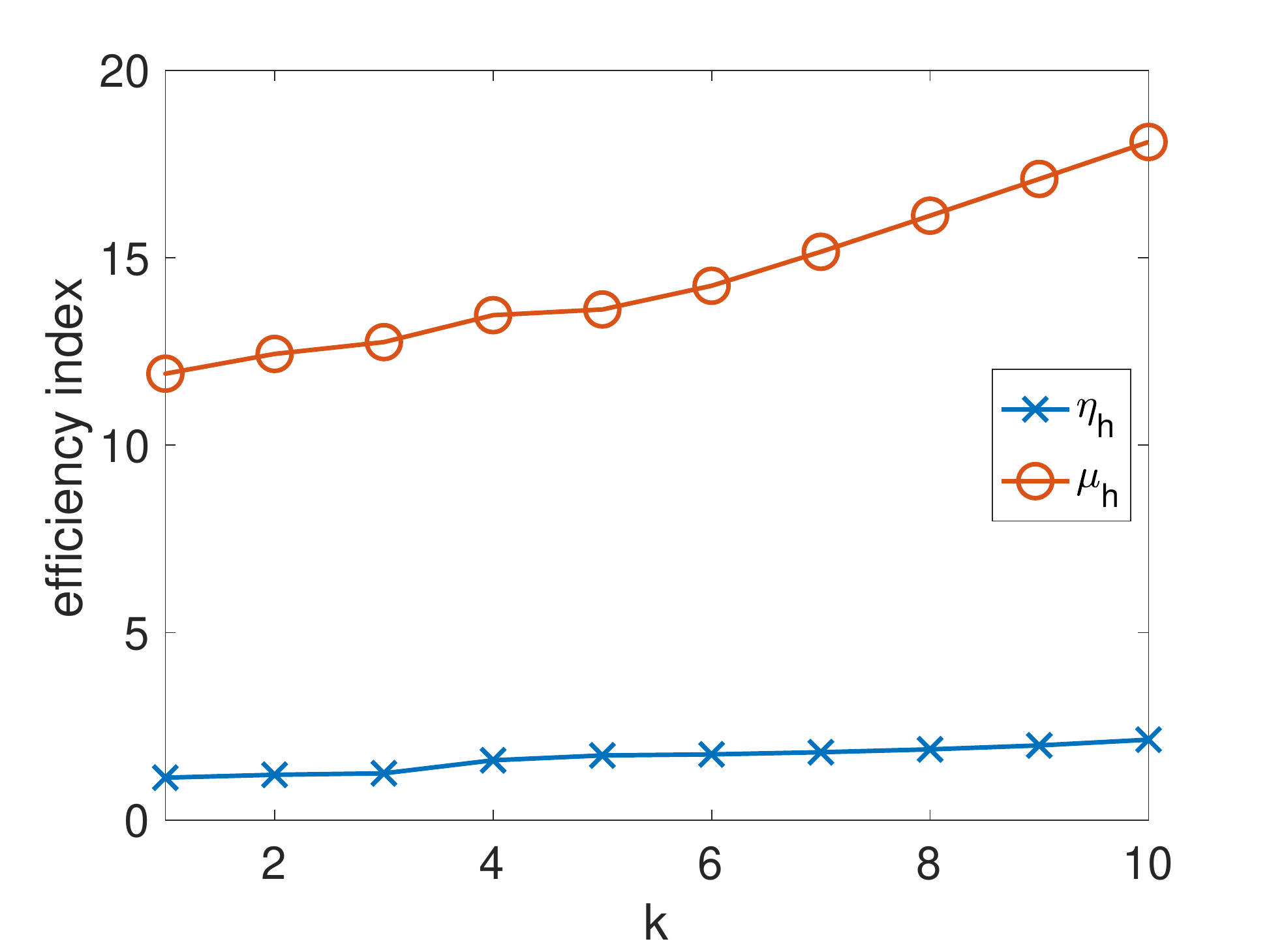}
\caption{Polynomial robustness of the equilibrated \textit{a posteriori} error estimator $\eta_h$ in comparison to the residual \textit{a posteriori} error estimator $\mu_h$
for the L-brick example.}
\label{fig:L-brick:robustness}
\end{figure}
As second example, we solve the homogeneous Maxwell problem
on the (nonconvex) domain 
\begin{align*}
\Omega = (-1,1)\times(-1,1)\times(0,1)\setminus \left(
  [0,1]\times[-1,0]\times[0,1]\right).
\end{align*}
As solution, we choose the singular function
\begin{align*}
\vu(x,y,z) = \nabla \times \left( \begin{array}{c} 0\\ 0 \\(1-x^2)^2(1-y^2)^2((1-z)z)^2r^{2/3}\cos(\frac{2}{3}\varphi) \end{array}\right),
\end{align*}
where $(r,\varphi)$ are the two dimensional polar coordinates in the $x$-$y$-plane, and we choose the right hand side $\vj$ accordingly.

Due to the edge singularity, uniform mesh refinement leads to suboptimal convergence rates of $\mathcal{O}(N_h^{-2/9})$, as shown in Figure~\ref{fig:L-brick}, left plot, for $k=2$. In contrast, adaptive mesh refinement leads to faster convergence rates. 
Note that, for this example, due to the edge singularity, anisotropic adaptive mesh refinement is needed to observe optimal convergence rates. Hence, the convergence rates in Figure~\ref{fig:L-brick} are limited by the employed isotropic adaptive mesh refinement. In Figure~\ref{fig:L-brick}, we observe the best possible rates for adaptive isotropic mesh refinement that is $\mathcal{O}(N^{-1/3})$ for $k=1$, $\mathcal{O}((N/\ln(N))^{-2/3})$ for $k=2$, and $\mathcal{O}(N^{-2/3})$ for $k\geq3$, cf. \cite[section 4.2.3]{A1999}. This shows experimentally that the adaptive algorithm generates meshes which are quasi-optimal for isotropic refinements. Again, the (global) efficiency indices are approximately between 1 and 2, as shown in Figure~\ref{fig:L-brick}, right plot.

Next, we compare the efficiency indices of $\eta_h$ for $k$-refinement to those of the residual \textit{a posteriori} error estimator \cite{beck00}
\begin{align*}
\mu_h^2 := \sum_{T\in\Th} \frac{h_T^2}{k^2} \| \vj - \nabla\times\vH_h \|_{T}^2
+ \sum_{f\in\Fhin}\frac{h_f}{k} \| \jut{ \vH_h } \|_{f}^2,
\end{align*}
with $h_T$ the diameter of element $T$ and $h_f$ the diameter of face
$f$. In Figure~\ref{fig:L-brick:robustness}, we observe the well known
fact that the residual \textit{a posteriori} error estimator $\mu_h$
is not robust in the polynomial degree $k$, whereas our new estimator
appears to be more robust in~$k$.

\subsection{Example with discontinuous permeability}
\begin{figure}[t]
\includegraphics[width=0.49\textwidth]{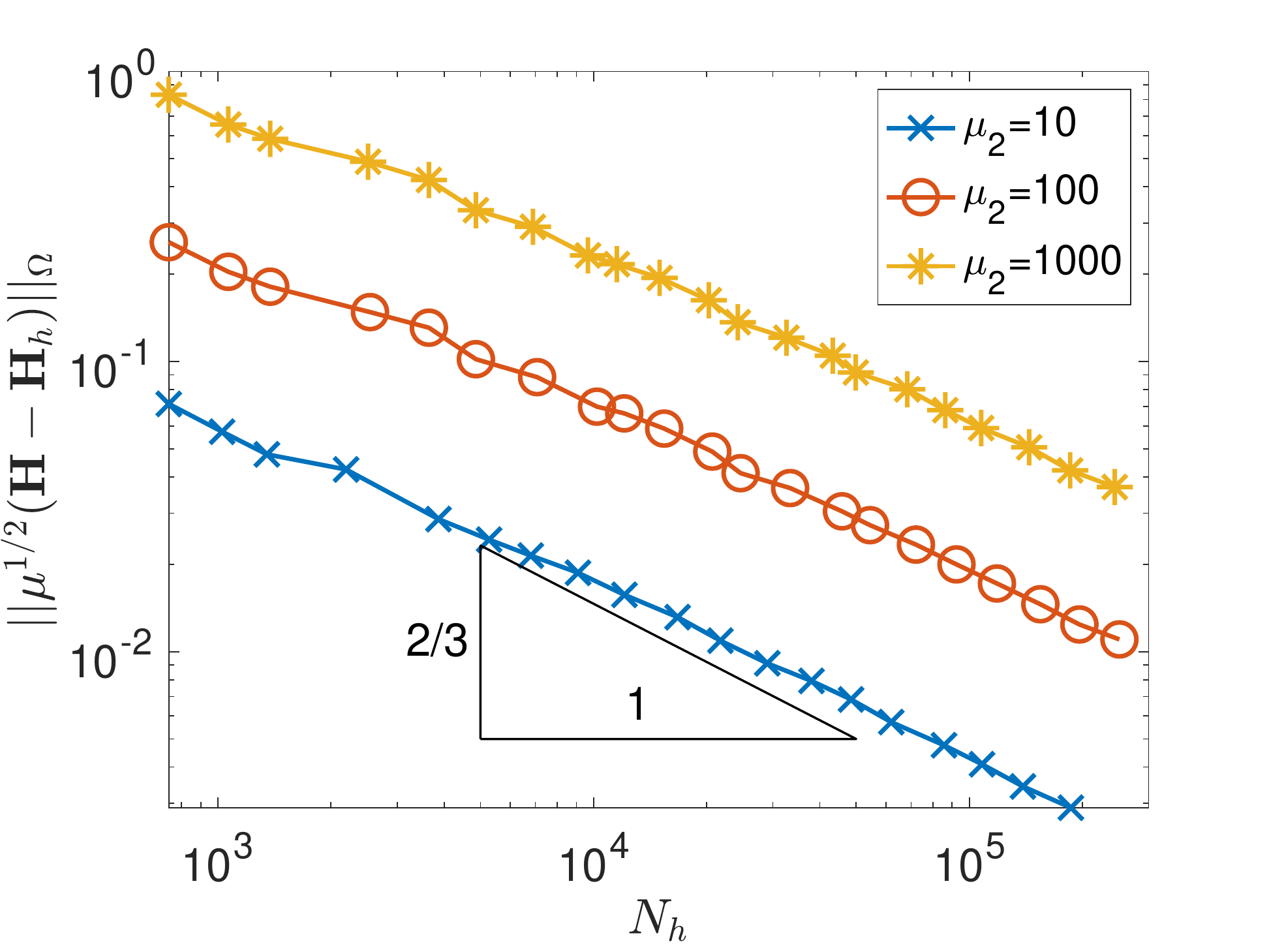}
\includegraphics[width=0.49\textwidth]{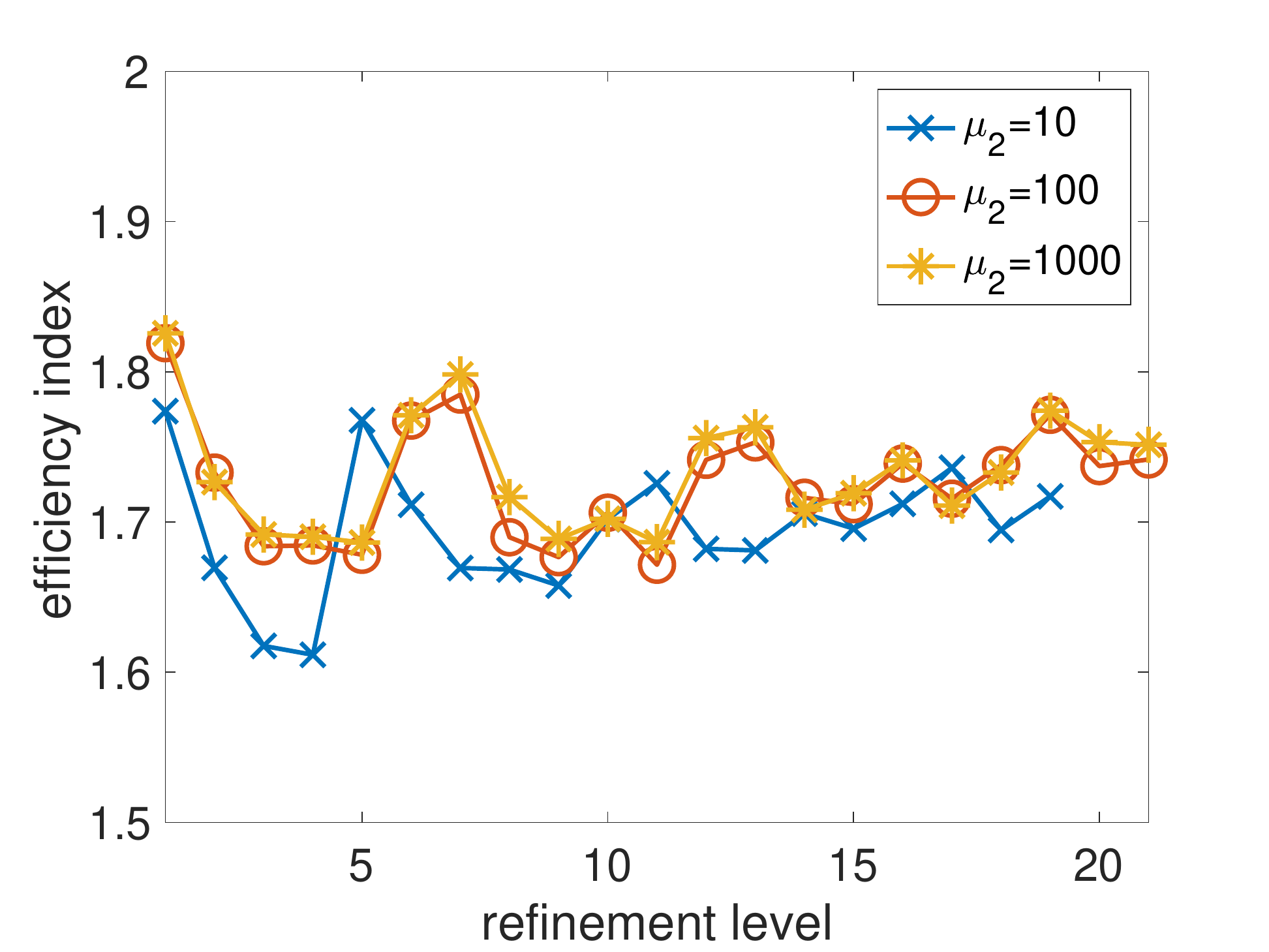}
\caption{Error and efficiency indices for adaptive mesh refinement for the example with discontinuous permeability.}
\label{fig:VarCoeff}
\end{figure}
\begin{figure}[t]
\includegraphics[width=0.49\textwidth]{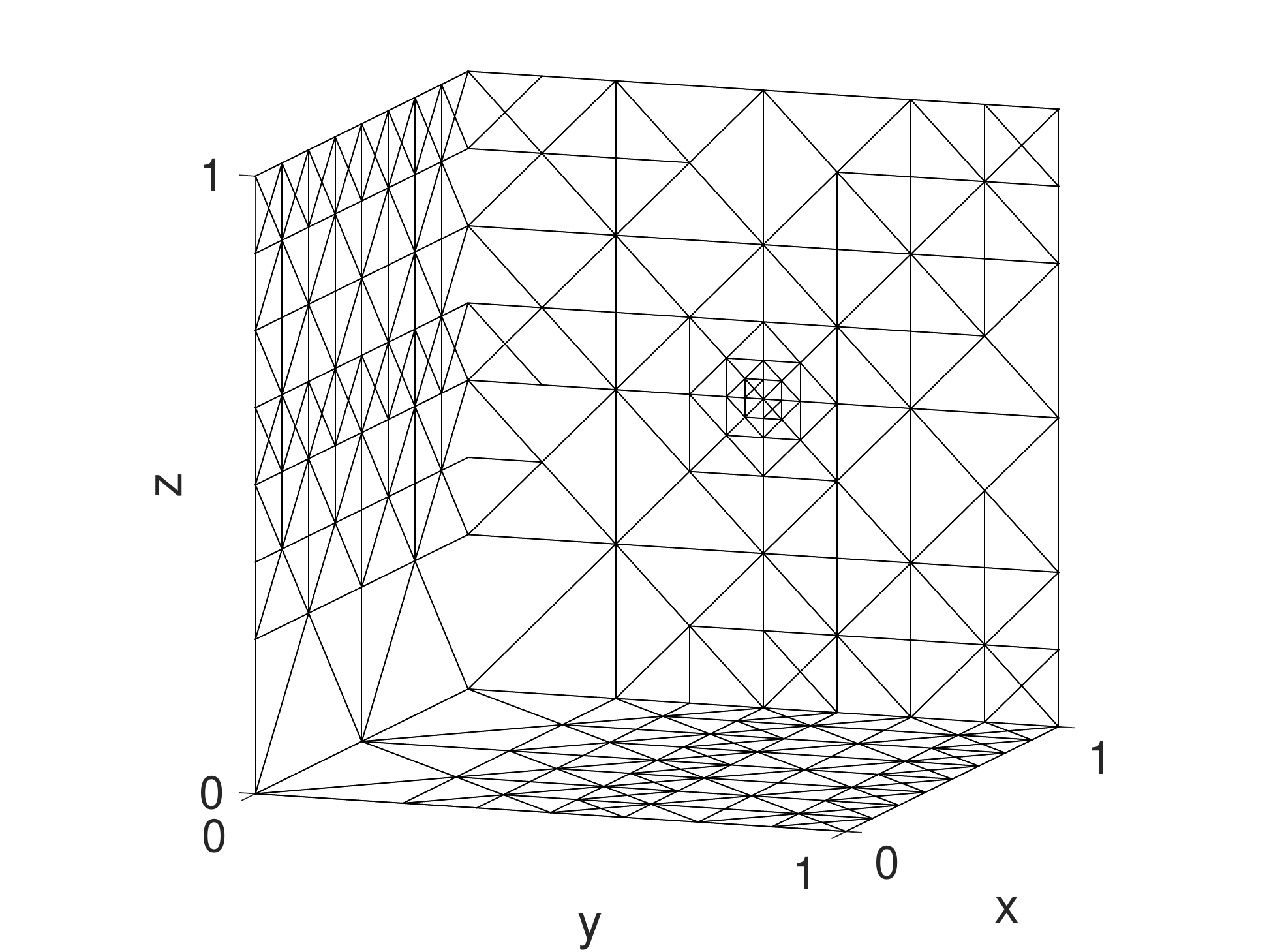}
\includegraphics[width=0.49\textwidth]{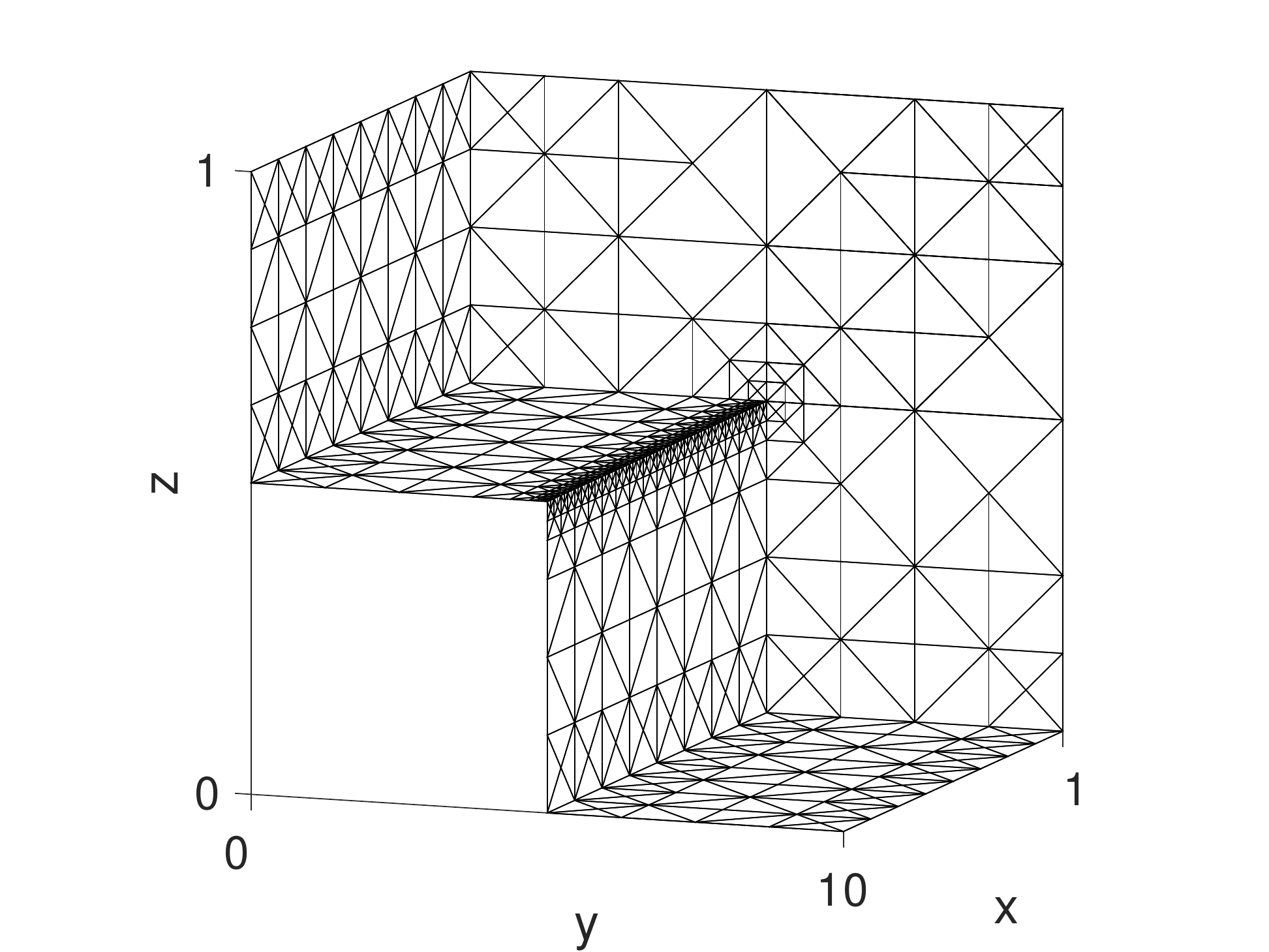}
\caption{Adaptive surface meshes of $\Omega$ (left) and of the subdomain of the coefficient $\mu_2=1000$ (right).}
\label{fig:VarCoeff:Mesh}
\end{figure}
In this example, we choose a discontinuous permeability 
\[
\mu(x,y,z) = \left\{
\begin{array}{ll}
\mu_1 & \text{if } y<1/2 \text{ and } z < 1/2,\\
\mu_2 & \text{otherwise}, 
\end{array}
\right.
\]
on the unit cube $\Omega = (0,1)^3$
and the right hand side $\mathbf{j}=(1,0,0)^t$.
We choose $\mu_1=1$ and vary $\mu_2= 10^\ell$ for $\ell=1,2,3$.
Since the exact solution is unknown, we approximate the error by comparing the numerical approximations
to a reference solution, which is obtained from the last numerical approximation by 8 more adaptive mesh refinements.
In these numerical experiments, we restrict ourselves to $k=2$.
As shown in Figure~\ref{fig:VarCoeff}, adaptive mesh refinement leads to optimal convergence rates for isotropic mesh refinement ($\mathcal{O}((N/\ln(N))^{-2/3})$),
and the efficiency indices are between 1 and 2, independently of the contrast of the discontinuous permeability.
In Figure~\ref{fig:VarCoeff:Mesh}, we display an adaptive mesh after 14 refinement steps with about $5\cdot 10^4$ degrees of freedom.
We observe strong adaptive mesh refinement towards the edge between the points $(0,1/2,1/2)^t$ and $(1,1/2,1/2)^t$.

\section{Conclusion}
\label{sec:conclusion}
We have presented a novel \textit{a posteriori} error estimator for
arbitrary-degree N\'ed\'elec elements for solving magnetostatic
problems. This estimator is based on an equilibration principle and is
obtained by solving only very local problems (on single elements, on
single faces, and on very small sets of nodes). We have derived a
constant-free reliability estimate and a local efficiency estimate,
and presented numerical tests, involving a smooth solution and a
singular solution, that confirm these results. Moreover, the numerical
results show an efficiency index between 1 and 2 in all considered
cases, also for large polynomial degrees $k$, and the dependence on $k$ appears to be small.
Some remaining questions are how to extend the proposed error estimator for domains with curved boundaries or for domains with a smoothly varying permeability.

\bibliographystyle{abbrv}
\bibliography{PosterioriError,Maxwell3dNumerics}

\appendix
\section{Error due to data projection}
\begin{thm}
\label{thm:projectionError}
Let $\vj\in H(\dvg^0;\Omega)\cap H^1(\Omega)^3$ be some divergence-free current distribution, let $\vu\in H_0(\curl;\Omega)\cap H(\dvg^0;\Omega)$ be the solution to
\begin{align*}
(\mu^{-1}\nabla\times \vu,\nabla\times\vw)_{\Omega} &= (\vj,\vw)_{\Omega} &&\forall \vw\in H_0(\curl;\Omega).
\end{align*}
and let $\vu' \in H_0(\curl;\Omega)\cap H(\dvg^0;\Omega)$ be the solution to
\begin{align*}
(\mu^{-1}\nabla\times \vu',\nabla\times\vw)_{\Omega} &= (\Pi_{\RT_k(\Th)}\vj,\vw)_{\Omega} &&\forall \vw\in H_0(\curl;\Omega),
\end{align*}
where $\Pi_{\RT_k(\Th)}$ denotes the standard Raviart--Thomas interpolation operator corresponding to the space $\RT_k(\Th)$ \cite{nedelec80}. Set $\vH:=\mu^{-1}\nabla\times\vu$ and $\vH':=\mu^{-1}\nabla\times\vu'$. If there exists an extension $\vj^*$ of $\vj$ to $\mathbb{R}^3$ such that $\vj^*\in H(\dvg^0;\mathbb{R}^3)\cap H^{k}(\mathbb{R}^3)^3$ for some $k\geq 1$, and such that $\vj^*$ has compact support, then  
\begin{align}
\label{eq:projectionError}
\| \mu^{1/2}(\vH-\vH') \|_{\Omega} &\leq Ch^{k+1}\| \vj^* \|_{H^k(\mathbb{R}^3)^3},
\end{align}
for some positive constant $C$ that does not depend on the mesh size $h$.
\end{thm}

\begin{proof}
In this proof, $C$ always denotes some positive constant that does not depend on the mesh size $h$.

Let $\vartheta\in\mathcal{C}_0^\infty(\mathbb{R}^3)$, with $\int_{\mathbb{R}^3}\vartheta(\vx) \;\dx=1$, be a smooth function with compact support and let $\vR: H(\dvg^0;\mathbb{R}^3)\rightarrow H^1(\mathbb{R}^3)^3$ be the regularised Poincar\'e integral operator given by
\begin{align*}
\vR\vj (\vx) &:= \int_{\mathbb{R}^3} \vartheta(\vz)\left( -(\vx-\vz)\times \int_{1}^\infty \tau\vj(\tau(\vx-\vz)+\vz) \;\mathrm{d}\tau \right) \;\dz \\
&= \int_{\mathbb{R}^3} -(\vx-\vy)\times\vj(\vy)\left( \int_{1}^\infty t(t-1)\vartheta(\vx+t(\vy-\vx)) \;\mathrm{d}t \right) \;\dy.
\end{align*}
This is exactly the operator $R_2$ of \cite[Definition 3.1]{costabel10}. From \cite[(3.14)]{costabel10}, it follows that $\nabla\times(\vR\vj^*)=\vj^*$ and from \cite[Corollary 3.4]{costabel10}, it follows that $\vR\vj^*\in H^{k+1}(\mathbb{R}^3)^3$ and $\|\vR\vj^*\|_{H^{k+1}(\mathbb{R}^3)^3}\leq C\|\vj^*\|_{H^{k}(\mathbb{R}^3)^3}$, where $\|\cdot\|_{H^k(\mathbb{R}^3)^3}$ denotes the standard norm corresponding to the Sobolev space $H^k(\mathbb{R}^3)^3$.
Furthermore, from \cite[Proposition 2, Remark 4]{nedelec86}, it follows that $\nabla\times\Pi_{\Nd_k^{(2)}(\Th)}(\vR\vj^*)=\Pi_{\RT_k(\Th)}(\nabla\times(\vR\vj^*))$, where $\Pi_{\Nd_k^{(2)}(\Th)}$ denotes the standard N\'ed\'elec interpolation operator corresponding to the curl-conforming N\'ed\'elec space of the second kind $\Nd^{(2)}_k(\Th)$~\cite{nedelec86}. Hence, $\nabla\times\Pi_{\Nd_k^{(2)}(\Th)}(\vR\vj^*)=\Pi_{\RT_k(\Th)}\vj^*$. We can then derive
\begin{align*}
\| \mu^{1/2}(\vH-\vH') \|_{\Omega}^2 &= (\mu^{-1}\nabla\times(\vu-\vu'),\nabla\times(\vu-\vu'))_{\Omega} \\
&= (\vj-\Pi_{\RT_k(\Th)}\vj, \vu-\vu')_{\Omega} \\
&= (\nabla\times(\vR\vj^* - \Pi_{\Nd_k^{(2)}(\Th)}(\vR\vj^*)),\vu-\vu')_{\Omega} \\
&= (\vR\vj^* - \Pi_{\Nd_k^{(2)}(\Th)}(\vR\vj^*), \nabla\times(\vu-\vu'))_{\Omega} \\
&= (\mu^{1/2}(\vR\vj^* - \Pi_{\Nd_k^{(2)}(\Th)}(\vR\vj^*)), \mu^{1/2}(\vH-\vH'))_{\Omega} \\
&\leq \| \mu^{1/2}(\vR\vj^* - \Pi_{\Nd_k^{(2)}(\Th)}(\vR\vj^*)) \|_{\Omega} \| \mu^{1/2}(\vH-\vH') \|_{\Omega} \\
&\leq Ch^{k+1}\| \vR\vj^* \|_{H^{k+1}(\Omega)^3} \| \mu^{1/2}(\vH-\vH') \|_{\Omega} \\
&\leq Ch^{k+1}\| \vj^* \|_{H^{k}(\mathbb{R}^3)^3} \| \mu^{1/2}(\vH-\vH') \|_{\Omega},
\end{align*}
where the sixth line follows from the Cauchy--Schwarz inequality and the seventh line follows from the interpolation properties of the $\Pi_{\Nd^{(2)}_k(\Th)}$ operator \cite[Proposition~3]{nedelec86}. Inequality \eqref{eq:projectionError} now follows immediately from the above.
\end{proof}

\end{document}